\documentclass[10pt]{amsart}

\usepackage{datetime}
\usepackage{etex}
\usepackage{amsmath, amssymb}
\usepackage{array}
\usepackage[frame,cmtip,arrow,matrix,line,graph,curve]{xy}
\usepackage{graphpap, color, paralist, pstricks}
\usepackage[mathscr]{eucal}
\usepackage[pdftex]{graphicx}
\usepackage[pdftex,colorlinks,citecolor=blue]{hyperref}
\usepackage{tikz}
\usepackage{enumitem}
\usepackage{comment}
\usepackage[backend=biber,style=alphabetic,url=false,doi=false,isbn=false,sorting=nyt,maxbibnames=9, giveninits]{biblatex}
\newtheorem{theorem}{Theorem}[section]
\newtheorem{proposition}[theorem]{Proposition}
\newtheorem{corollary}[theorem]{Corollary}
\newtheorem{lemma}[theorem]{Lemma}

\newtheorem{conjecture}[theorem]{Conjecture}

\theoremstyle{definition}
\newtheorem{definition}[theorem]{Definition}

\newtheorem{question}[theorem]{Question}
\newtheorem{remark}[theorem]{Remark}

\newtheorem{assumption}[theorem]{Assumption}
\newtheorem{algorithm}[theorem]{Algorithm}
\newtheorem*{acknowledgement}{Acknowledgement}

\newcommand{\PP}{\mathbb{P}}
\newcommand{\QQ}{\mathbb{Q}}
\newcommand{\CC}{\mathbb{C}}
\newcommand{\RR}{\mathbb{R}}
\newcommand{\ZZ}{\mathbb{Z}}

\newcommand{\cO}{\mathcal{O} }

\newcommand{\cC}{\mathcal{C} }

\newcommand{\cS}{\mathcal{S} }

\newcommand{\wtS}{\widetilde{\mathcal{S}}}

\newcommand{\rH}{\mathrm{H} }
\newcommand{\rM}{\mathrm{M} }

\newcommand{\ee}{\mathbf{e}}

\newcommand{\proj}{\mathrm{Proj}\;}

\def\sym{\mathrm{Sym} }
\def\Hom{\mathrm{Hom} }

\newcommand{\intr}{\mathrm{int}\,}
\newcommand{\Conv}{\mathrm{Conv}}

\newcommand{\Wedge}{\mathsf{\Lambda}}
    
\def\SL{\mathrm{SL}}

\def\GL{\mathrm{GL}}
\def\PGL{\mathrm{PGL}}
\def\SO{\mathrm{SO}}
\def\Spin{\mathrm{Spin}}
\def\Gr{\mathrm{Gr}}
\def\Sp{\mathrm{Sp}}

\def\git{/\!/ }
\def\Sym{\mathrm{Sym}}

\begin{document}

\title{Computation of GIT quotients of semisimple groups}
\date{\today, \currenttime}

\author{Patricio Gallardo}
\address{Department of Mathematics, 
University of California, Riverside, CA, 92521}
\email{pgallard@ucr.edu}  

\author{Jesus Martinez-Garcia}
\address{Department of Mathematical Sciences, University of Essex, Colchester, Essex CO4 3SQ, United Kingdom}
\email{jesus.martinez-garcia@essex.ac.uk}

\author{Han-Bom Moon}
\address{Department of Mathematics, Fordham University, New York, NY 10023}
\email{hmoon8@fordham.edu}

\author{David Swinarski}
\address{Department of Mathematics, Fordham University, New York, NY 10023}
\email{dswinarski@fordham.edu}

\maketitle

\begin{abstract}
We describe three algorithms to determine the stable, semistable, and torus-polystable loci of the GIT quotient of a projective variety by a reductive group. The algorithms are efficient when the group is semisimple. By using an implementation of our algorithms for simple groups, we provide several applications to the moduli theory of algebraic varieties, including the K-moduli of algebraic varieties, the moduli of algebraic curves and the Mukai models of the moduli space of curves for low genus. We also discuss a number of potential improvements and some natural open problems arising from this work.
\end{abstract}

\section{Introduction}\label{sec:intro}

Group actions and orbit spaces are ubiquitous in mathematics. The existence of symmetry in a given object oftentimes enables us to prove a surprising number of rich and deep results for them. Representation theory of finite groups and classical groups is one of excellent and approachable examples of this slogan. In geometry and topology, many interesting spaces are constructed as the orbit space (or quotient space) of another space by a symmetry group. For example, any hyperbolic surface can be obtained by a quotient space of the hyperbolic plane and the moduli space of Riemann surfaces is a quotient space of the Teichm\"uller space by the mapping class group action. 

In algebraic geometry, one often needs to construct the quotient space of an algebraic variety under a group action, while preserving a nice algebraic structure. There are several constructions, including the Chow quotient and the Hilbert quotient \cite{Kap93}. However, in applications where the group involved is reductive,  the most widely used quotient construction is the Geometric Invariant Theory (GIT) quotient, developed by Mumford \cite{MFK94}. There are two prominent reasons why this construction is widely used. The computation of the GIT quotient is approachable in many interesting examples, due to the Hilbert-Mumford criterion (Theorem \ref{thm:HilbertMumford}). The second reason is that if the given variety is projective, the quotient variety is also projective. Many interesting algebraic varieties, including moduli spaces of varieties and sheaves, have been constructed in this manner. 

\subsection{Main results}

The main goal of this article is to provide efficient computational algorithms and their implementations to compute the GIT quotient of a projective variety by a reductive group; with emphasis on the case where the group is semisimple.

To describe the GIT quotient of an algebraic variety $X$, one needs to describe two important open subsets, the so-called \emph{semistable locus} $X^{ss}$ and the \emph{stable locus} $X^s$ (for the definition and why they are essential, see Section \ref{sec:GITquotient}). The GIT quotient $X\git G$ is not the quotient of the whole $X$, but its open subset $X^{ss}$. The `quotient map' $X^{ss} \to X\git G$ is the set theoretic quotient map only over the open subvariety $X^s \subset X^{ss}$. In principle, these loci can be computed by employing the aforementioned Hilbert-Mumford criterion. However, the computation typically involves highly non-trivial convex geometry calculations, and as a result, many GIT analyses employ computer-assisted calculations. To our knowledge, in the literature, these calculations have been carried out on an \emph{ad hoc} basis; typically, each group of authors wrote a new computer program to analyse each new GIT problem.  One of our long-term goals is to completely automate these calculations. As a first step, we clearly describe algorithms to perform three key steps in a GIT analysis, and implement them in \texttt{SageMath}.\footnote{We chose \texttt{SageMath} in the hope that our code will remain useful to the community for a long time.} This permits us to run many examples using one program and compare the performance of the algorithm as the input varies, we believe for the first time.

\begin{theorem}\label{thm:mainthm}
Let $(X, L)$ be a pair of a projective variety $X$ and a very ample line bundle $L$. Let $G$ be a semisimple group and suppose $X$ admits an $L$-linearized $G$-action. The finite list $P_{s}^{F}$ (resp. $P_{ss}^{F}$) of states (see definition in Corollary \ref{cor:finite}) that determines $X^{s}$ (resp. $X^{ss}$) can be calculated by using Algorithm \ref{alg:Stable} (resp. Algorithm \ref{alg:Semistable}). 
\end{theorem}

We describe the definition of a state and the meaning of `calculating' $X^{s}$ and $X^{ss}$ in Section \ref{sec:GITquotient}. As a matter of fact, we may replace the term \emph{projective variety} in Theorem \ref{thm:mainthm} for the more general \emph{projective scheme over $k$}. However we will keep it as is for simplicity.

For the study of moduli spaces of degenerated objects, it is also helpful to study the stratification of the quotient of strictly semistable locus $(X \git G) \setminus (X^s /G)$. Such stratification can be understood by the \emph{polystable locus} $X^{ps} \subset X^{ss} \setminus X^s$, insofar the GIT `boundary'  $(X\git G) \setminus (X^s/G)$ represents, as a set, the set of polystable orbits. 

To describe the stratification, it is necessary to describe a similar stratification on $(X \git T) \setminus (X^s/T)$ for the induced maximal torus $T$-action (see Section \ref{ssec:GITboundary} for the notation and background). Algorithm \ref{alg:Polystable} describes a systematic way to compute the latter.

\begin{theorem}\label{thm:mainthmpoly}
Let $(X, L)$ be a pair of a projective variety $X$ and a very ample line bundle $L$. Let $G$ be a semisimple group and suppose $X$ admits an $L$-linearized $G$-action. Let $T$ be a fixed maximal torus of $G$. The finite list $P_{ps}^{F}$ of states that determine $T$-polystable locus in $X^{ss} \setminus X^{s}$ can be calculated by using Algorithm \ref{alg:Polystable}. 
\end{theorem}

Our algorithms work for any reductive group. However, we expect that for a general non-semisimple reductive group (e.g. the case of a torus $T$) the algorithm is slow because of the nature of the problem. In particular, for a non-semisimple reductive group, our algorithm characterizing the semistable locus does not seem to have any advantage compared to that of Popov \cite[Appendix C]{DK15}. Consult Remarks \ref{rem:reductivestable} and \ref{rem:reductivesemistable}.

\subsection{Applications to moduli theory}

Our motivation for this project is to automate part of the work required to describe compact moduli spaces. As previously hinted, the usual approach to use GIT to describe the objects classified in a given moduli space is as follows: one finds a projective scheme $H$ where each point represents an object in the moduli space. For example, if one is interested in describing the moduli space of cubic surfaces, one may consider
\[
    H=\mathbb P^{19}\cong \mathbb P \rH^0(\mathbb P^3, \mathcal O_{\mathbb P^3}(3))^*
\]
which parameterizes cubic surfaces, since the scheme $H$ characterizes homogeneous polynomials of degree $3$ in $4$ variables.
However, two different objects in $H$ might be equivalent in the moduli space. Often, $H$ has a natural $G$-action so that two objects are equivalent if and only if they are equivalent up to the action of $G$ (in the above example for cubics, one may consider $G$ to be $\mathrm{PGL}_{4}$). Thus, one wants to consider the GIT quotient $H^{ss}\git G$, where $H^{ss}\subset H$ is the largest subset for which the quotient is an algebraic variety. Our methods (and software) will provide a finite list of deformation families of objects that describe, among other things, $H\setminus H^{ss}$. The specific way of representing these families by our software may not be very informative, so the geometer will still have to interpret the program output into geometric terms. The latter may not be a trivial matter at all, but a subtle problem in singularity theory. For instance, in the example of cubic surfaces, the program's output will describe the families as polynomials, which the geometer will still have to translate into geometric terms by describing the possible singularities of those families of polynomials. See  Section \ref{sec:cubic} for this example in detail.

By using an implementation of our algorithms in \texttt{SageMath} \cite{sagemath}, we recover many known computational results and obtain some new results in moduli theory. We suppress any technical details in the introduction and refer the reader to Section \ref{sec:cubic} for a worked-out example on cubic surfaces and Section \ref{sec:examplesandstatistics} for other results on many more examples, as well as to Section \ref{sec:Kstability} for one example on the moduli of anti-canonical curves in a quadric surface (which is later reinterpreted as the family 2.24 in Mori-Mukai's classification of Fano threefolds). %

\subsection{Applications to K-stability}
The setting for the moduli of cubic surfaces above can clearly be generalized to that of hypersurfaces in projective space, or more generally complete intersections. Since the complexity in the analysis of the output increases with the degree and the dimension (the larger their degree is, polynomials may have nastier singularities), the most accessible applications will be in lower degrees, i.e. in the realm of Fano varieties. In recent years it has become apparent that Fano varieties admit a projective compactification thanks to the theory of K-stability. The latter is an algebro-geometric stability notion that controls the singularities of \emph{all} $\CC^*$-equivariant degenerations of an algebraic variety over the germ of a curve. This relatively recent theory first emerged from analytic geometry when considering the Calabi problem on projective manifolds of positive Ricci curvature (i.e. Fano manifolds), i.e. the problem of the existence of K\"ahler-Einstein metrics on these manifolds. It follows from \cite{CDS}, cf. \cite{Tian2015, Tian2015Corrigendum}, that a smoothable Fano variety admits a K\"ahler-Einstein metric if and only if it is K-polystable. There are further generalizations of this result (in the most general statement it is known as the \emph{Yau-Tian-Donaldson conjecture}), but the stated one is enough for our purposes. 

Due to the number of degenerations to consider in the definition of K-stability, determining when a Fano variety is K-polystable is just as challenging as determining whether it admits a K\"ahler-Einstein metric. However, here  moduli theory can come in handy. It is known that K-polystable smoothable Fano varieties form a projective moduli space known as the K-moduli space \cite{Odaka2014, Liu-Xu-Zhuang}. Yet, even in dimension 3 (the highest dimension for which smooth Fano varieties are classified \cite{Iskovskikh1, Iskovskikh2, Mori-Mukai1, Mori-Mukai2, Mori-Mukai3}), a systematic approach to determining K-stable Fano manifolds was not attempted until recently \cite{CalabiFanoProject} and knowledge of the K-moduli is even more lacking---only a few of the connected components have been studied \cite{Liu-Xu-cubic3folds, Spotti-Sun-dP4}.  Relying on our construction, we can recover a recent result of Papazachariou, who used an \emph{ad hoc} GIT computation to describe the connected component of the  K-moduli for family 2.25 in the Mori-Mukai classification.

\begin{theorem}[{\cite{papazachariou2022k}}]
\label{thm:Kstabilityintro}
The compact component of the K-moduli space of smooth Fano threefolds corresponding to family 2.25 in the Mori-Mukai classification is canonically isomorphic to the GIT quotient
\[
    \PP(\bigwedge^2 \rH^0(\mathbb P^3, \mathcal O_{\mathbb P^3}(2))^*)\git \SL_4,
\]
which parametrizes orbits of complete intersections of two quadrics in $\mathbb P^3$.
\end{theorem}

Theorem \ref{thm:Kstabilityintro} is a proof-of-concept for an approach that uses GIT to describe the K-moduli of Fano threefolds and the role GIT can play in it. Indeed, given that \cite{CalabiFanoProject} and subsequent work have pretty much completed the classification of the general element of the K-moduli of Fano threefolds, one can now construct a GIT compactification and, for every element represented in the GIT compactification whose K-stability is unknown, apply methods from \cite{CalabiFanoProject} to determine it. The K-unstable elements (if any) in the GIT compactification have to be removed and replaced by others which are K-polystable. Then standard methods (moduli continuity method in \cite{OdakaSpottiSun}, cf. \cite{GMGS21}, or reverse moduli continuity method in \cite{papazachariou2022k}) can be used to find an isomorphism between the modified GIT quotient and the K-moduli. Given that the smooth locus of the K-moduli of Fano threefolds is now almost complete \cite{CalabiFanoProject}, its full description is within reach. We believe that this work will provide the technical GIT cornerstone to apply the (reverse) moduli continuity method to each of the families.

\subsection{Birational models of the moduli space of stable curves}

The moduli spaces ${\rM}_g$ of smooth curves and their compactifications  $\overline{\rM}_g$ are some of the most intensively studied moduli spaces in algebraic geometry. In a series of papers, Mukai described non-compact birational models of ${\rM}_g$ with $7 \le g \le 9$ as quotients of open dense subsets of symmetric spaces \cite{Mukai1992survey,Mukai1992g8,Mukai1995,Mukai2010}. By taking the GIT quotients of these symmetric spaces, we obtain compactifications of Mukai's spaces, and they are projective birational models of $\overline{\rM}_g$. We analyzed (semi)-stability for the corresponding GIT problems. For $g = 7$, the GIT problem is too large to analyze stability in full detail (see Section \ref{subsec:Mukai}). However, some geometric results were described in the fourth author's recent preprint \cite{Swin23}. Interestingly, the output of our algorithms is the simplest for $g = 9$, and the full (semi-)stability is described in \cite{sagemath-code}.

\subsection{Weyl group symmetry}

One key ingredient to our approach is to investigate the Weyl group symmetry carefully. For a polarized projective variety $(X, L)$ equipped with a linearized $G$-action, the GIT quotient can be described by using the induced $G$-action on the $G$-representation $V := \rH^{0}(X, L)$. For a fixed maximal torus $T$ of $G$, the $T$-stable/semistable loci $X^s$ and $X^{ss}$ can be described by using the finite set of characters $\Xi_V$ of $T$ on $V$. If we restrict ourselves to semisimple groups, $\Xi_V$ has Weyl group symmetry, so we can reduce the set $\Xi_V$ to proper subsets of \emph{essential characters} $\Xi_V^{E, s}$ and $\Xi_V^{E, ss}$. For the computation of $P_s^F$ and $P_{ss}^F$, the biggest bottleneck is considering many subsets of $\Xi_V$, and reducing $\Xi_V$ to $\Xi_V^{E,s}$ and $\Xi_V^{E,ss}$ therefore provides a significant improvement.

\subsection{Related work}

In \cite[Appendix C]{DK15}, Popov also describes algorithms for studying GIT quotients.  Here we briefly explain the difference between  his work and ours. Popov's algorithm computes a  stratification of the null-cone, which is the complement of the semistable locus; in this work, we also provide the stable and $T$-polystable loci computation. Also, Popov's algorithm calculates all unstable strata, whereas  Algorithm \ref{alg:Semistable} focuses on the maximal unstable strata only. While both his approach and ours can be applied to general reductive group actions, ours is more efficient for semisimple groups, since it makes use of the symmetry of the Weyl group to reduce the bottleneck of the algorithm. Even for the semistable locus of a non-semisimple reductive group action, for our purposes, our algorithm will be slightly more efficient than Popov's, because we only compute the maximal unstable strata.

In \cite{Der99, DK08}, the authors provide an algorithm, based on Gr\"obner basis techniques, to find the invariant subring of a given coordinate ring. Our algorithm does not compute any explicit invariants --- it only  detects whether there is a non-vanishing invariant for each point or not. Since Gr\"obner basis calculations can be very expensive in terms of time and memory, this approach is not suitable for many of the examples we wish to study.

There are some other works on computational GIT. For a fixed algebraic variety and a group action, the change of its linearization may provide different GIT quotients \cite{DH98, Tha96}. For the torus action on an affine variety, an algorithm to keep track of the variation is described in \cite{Kei12, BKR20}. Since any Mori Dream Space can be obtained in this way \cite[Proposition 2.9]{HK00}, it has important implications to the birational geometry of algebraic varieties, in particular Fano varieties \cite{LMR20}. However, this direction of research does not have any significant overlap with the contents of this article.

\subsection{Organization of the paper}
This article is intended to attract readers from various backgrounds. Up to section \ref{sec:examplesandstatistics}, only minimal prerequisites on algebraic geometry and representation theory of classical groups are assumed. In particular, experts in GIT may want to skip most of section \ref{sec:GITquotient}. The last two sections are devoted to advanced applications in moduli theory. 

Section \ref{sec:GITquotient} gives a definition and fundamental properties of GIT quotient. In Section \ref{sec:algorithm}, we describe our algorithms for the stable/semistable loci computation. Section \ref{sec:cubic} deals with a classically non-trivial well-known example (going as far back as Hilbert's work in the 19th century \cite{Hil93}) to demonstrate in a simple case how our algorithm works. In the remaining sections, we provide some statistics on the algorithms' running times and complexity (Section \ref{subsec:stat}),  and discuss consequences in the compactification of the moduli space of hypersurfaces (sections \ref{subsec:quintic-3folds}, \ref{subsec:cubic-5folds}), the birational geometry of the moduli spaces of curves (Section \ref{subsec:Mukai}), and the theory of the moduli space of K-stable objects (Section \ref{sec:Kstability}). The last section discusses some possible improvements to the algorithms and open problems for future work.%

For most of the paper, we work on an algebraically closed field $\Bbbk$ of arbitrary characteristic, with the exception of Section \ref{sec:Kstability}, where we assume that the base field is $\CC$. The algorithms work for any projective scheme, including reducible or non-reduced ones. They also work over non-algebraically closed fields and even relative bases, if the algebraic group scheme is split over the base \cite{Ses77}. But in order to simplify the exposition, we do not pursue full generality.

\begin{acknowledgement}
	This work was partially supported by a SQuaRE grant of the American Institute for Mathematics (AIM) which allowed the authors to meet several times, at AIM headquarters and online, to carry out this work. We thank AIM for their support and their patience with us during the pandemic years.
 
    JMG is partially supported by an EPSRC grant EP/V055399/1. We also received partial support from the University of Essex Department of Mathematical Sciences Research and Innovation Fund. 
    
    We would like to thank Tiago Duarte Guerreiro and Theodoros Papazachariou for useful discussions. We would also like to thank Ian Morrison and Ruadha\'i Dervan for constructive comments on earlier versions of the manuscript.
\end{acknowledgement}

\section{GIT quotients}\label{sec:GITquotient}

In this section, we review key definitions and results of GIT and fix notation. Standard references are \cite{MFK94}, \cite{Dol03}, and \cite{DK15}.

\subsection{Definition of projective GIT quotient}\label{ssec:GITquotient}

Let $(X, L)$ be a pair of a projective variety $X$ and a very ample line bundle $L$. This is equivalent to have an embedding $X \hookrightarrow \PP^{r} \cong \PP \rH^{0}(X, L)^{*}$ (here $r = \dim \rH^{0}(X, L) - 1$). We are interested in a good algebraic group action on $X$. 

A linear algebraic group $G$ is \emph{reductive} if its maximal smooth connected solvable normal subgroup is a torus. This class contains many important examples of groups, including finite groups, tori, and many classical algebraic groups, such as $\GL_n$, $\SL_n$, $\mathrm{O}_n$, $\mathrm{SO}_n$, and $\mathrm{Sp}_n$. Moreover, products, finite extensions, and quotients of reductive groups are also reductive. 

A \emph{semisimple} group is an algebraic group such that every smooth connected solvable normal subgroup is trivial. Thus, all semisimple groups are reductive. Examples include $\SL_n$, $\mathrm{SO}_n$, $\mathrm{Sp}_n$, and their direct sums, finite extensions and quotients. Hence $\mathrm{PGL}_n$ is semisimple, too. Semisimple groups can be classified by analyzing their Lie algebras. The Lie algebra of a semisimple group is a direct sum of simple Lie algebras, and the simple Lie algebras are classified by their Dynkin types ($A_n$, $B_n$, $C_n$, $D_n$, $E_6$, $E_7$, $E_8$, $F_4$, and $G_2$ \cite[Chapter 21]{FH91}). By Remark \ref{rmk:Liealg}, we may assume that $G$ is a product of simple groups.

Let $G$ be a reductive group. Suppose that $G$ acts on $X$ and assume further that this $G$ action can be extended to $L$ (i.e. the $G$-action on $X$ is \emph{linearized} to $L$). Then, for each $m \ge 0$, $\rH^{0}(X, L^{m})$ is a finite-dimensional $G$-representation. We denote by $\rH^{0}(X, L^{m})^{G}$ the subspace of $G$-invariant vectors. 

Let 
\[
	R(X, L) := \bigoplus_{m \ge 0}\rH^{0}(X, L^{m})
\]
be the section ring of $L$. Since $L$ is a very ample line bundle on $X$, $X = \proj R(X, L)$. Because the $G$-action is linearized, $R(X, L)$ has an induced $G$-action. Indeed the invariant subset 
\[
	R(X, L)^{G} := \bigoplus_{m \ge 0}\rH^{0}(X, L^{m})^{G}
\]
has a sub graded ring structure. 

Recall that $R(X, L)$ is the ring of `coordinate functions' of $X$. If there is a good quotient variety $X/G$, then its ring of coordinate functions should be identified with the $G$-invariant coordinate functions of $X$. This motivates the following definition. 

\begin{definition}\label{def:GITquotient}
The \emph{GIT quotient of $X$} (with respect to $L$ and the $G$-action on $L$) is defined by 
\[
	X \git_{L}G := \proj R(X, L)^{G}.
\]
As $G$ is reductive, by Nagata's theorem \cite[Theorem 3.3]{Dol03} (for positive characteristics, see \cite{Hab75}, \cite{Ses77}), $R(X, L)^{G}$ is also a graded finitely generated $\Bbbk$-algebra, so $X\git_{L}G$ is a projective variety. 
\end{definition}

If there is no chance of confusion, then we drop the subscript $L$ and write $X \git G$. 

\begin{remark}
In the literature, the choice of $L$ and the extended $G$-action on $L$ is called a \emph{linearization}. The GIT quotient depends on a choice of a linearization, and if we choose a different very ample line bundle $L$, or a different extension of $G$-action to $L$, the quotient may change. %
See \cite{DH98,Tha96} for details.
\end{remark}

\subsection{Stability and semi-stability}\label{ssec:stability}

The GIT quotient $X \git_{L} G$ is different from the quotient $X/G$ in the category of topological spaces in two ways. First of all, $X \git_{L}G$ is not the quotient of the whole $X$, but that of an open subset of $X$. From the embedding $R(X, L)^{G} \hookrightarrow R(X, L)$, we obtain a map 
\[
	\pi : X = \proj R(X, L) \dashrightarrow \proj R(X, L)^{G} = X\git_{L}G. 
\]
However, in most cases, $\pi$ is not a regular map, but a rational map. Indeed, for $x \in X$, let $m_{x}$ be the associated homogeneous maximal ideal of $R(X, L)$. Then the image $\pi(x)$ is a point associated to $m_{x} \cap R(X, L)^{G}$, but it may be the irrelevant ideal $\bigoplus_{m > 0}R(X, L)^{G}$, which does not correspond to any point on $X\git_{L}G$. This observation leads to the following definition.

\begin{definition}\label{def:semistablepoint}
A point $x \in X$ is called \emph{semi-stable} if there is a $G$-invariant section $s \in \rH^{0}(X, L^{m})^{G}$ for some $m > 0$ such that $s(x) \ne 0$. Let $X^{ss}(L)$ be the set of semi-stable points on $X$. (If the choice of the linearization $L$ is not ambiguous, we often simplify the notation to $X^{ss}$.) 
\end{definition}

The set $X^{ss}(L)$ is open. If $x \in X^{ss}(L)$, then $m_{x} \cap R(X, L)^{G}$ is not an irrelevant ideal. This implies that there is a $G$-invariant section which does not vanish at $x$. Thus, we have a regular morphism $\pi : X^{ss}(L) \to X\git_{L}G$, which is clearly $G$-invariant. 

A second issue that arises in GIT (compared to quotients in the category of topological spaces) is that $X \git_{L}G$ may not be the orbit space of $X^{ss}(L)$, because some of the orbits are identified on $X \git_{L}G$. This is because a reductive group $G$ is not compact if it is positive dimensional, so the $G$-orbits are often not closed, hence the closure of an orbit may contain another orbit. These two orbits must be identified in the quotient, to obtain a separated quotient variety. 

\begin{definition}\label{def:stablepoint}
A point $x \in X$ is called \emph{stable} if:
\begin{enumerate}
\item there is a section $s \in \rH^{0}(X, L^{m})^{G}$ for some $m > 0$ such that $s(x) \ne 0$ (i.e. $x$ is semistable),
\item the orbit $Gx$ has the same dimension as $G$, and
\item the orbit $Gx \subset X_{s} = \{y \in X\;|\; s(y) \ne 0\}$ is closed.
\end{enumerate}
Let $X^{s}(L)$ be the set of stable points on $X$. (If the choice of the linearization $L$ is not ambiguous, we often simplify the notation to $X^{s}$.)  
\end{definition}

The subset $X^{s}(L)\subseteq X^{ss}(L)$ is open. The restriction of $\pi$ to $X^{s}(L)$ is now a genuine quotient map, and $\pi(X^{s}(L))$ is precisely the set of $G$-orbits in $X^{s}(L)$. On $X^{ss}(L) \setminus X^s(L)$, the map $X^{ss}(L) \to X \git_L G$ is not a set-theoretic quotient map, as several orbits can be identified to a single point. However, for each point $y \in X \git_L G \setminus X^{s}(L)\git_L G$, there is a unique \emph{closed} orbit $Gx \subset X^{ss}(L)$ such that $\pi (Gx) = y$. Such a point $x$ is called a \emph{strictly polystable point}. In other words, the description of the \emph{GIT boundary} $X\git_L G \setminus X^s/G$ is equivalent to the classification of strictly polystable points. 

The following notions complete the picture:

\begin{definition}\label{def:unstable}
A point $x \in X$ is \emph{unstable} if $x \in X \setminus X^{ss}(L)$. A point $x \in X$ is \emph{non-stable} if $x \in X \setminus X^{s}(L)$. The set of unstable points and the set of non-stable points are denoted by $X^{us}(L)$ and $X^{ns}(L)$, respectively --- or $X^{us}$ and $X^{ns}$ when no confusion is likely.
\end{definition}

We close this section with the following observation, which reduces the computation of the (semi-)stable locus to that of the ambient projective space. Set $V := \rH^{0}(X, L)$. Then, since $L$ is very ample, $\iota : X \hookrightarrow \PP V^{*}$. Furthermore, $\PP V^{*}$ has an induced $G$-action on $\cO(1) = \cO_{\PP V^{*}}(1)$ because $\rH^{0}(\PP V^{*}, \cO(1)) \cong \rH^{0}(X, L)$. Thus we may consider another GIT quotient $\PP V^{*}\git_{\cO(1)} G$. The map $\iota$ is $G$-equivariant. The next theorem tells us that the (semi-)stable locus is also compatible. 

\begin{theorem}[\protect{\cite[Theorem 1.19]{MFK94}}]\label{thm:functoriality}
Under the above situation, $X^{ss}(L) = X \cap \PP V^{* ss}(\cO(1))$ and $X^{s}(L) = X \cap \PP V^{* s}(\cO(1))$. 
\end{theorem}

Thus, the map $\iota$ induces the morphism between GIT quotients $X\git_{L}G \hookrightarrow \PP V^{*}\git_{\cO(1)}G$.

\subsection{Hilbert-Mumford criterion}

One of the many reasons why the GIT quotient is useful when compared to other algebro-geometric quotients (--- )for example the Chow quotient \cite{Kap93}) is that we may describe the quotient explicitly by calculating the (semi/poly-)stable locus. A key tool for this is the Hilbert-Mumford criterion, which provides a way to describe the (semi-)stable locus explicitly and combinatorially. 

By Theorem \ref{thm:functoriality}, we may assume that $X = \PP V^{*}$ where $V$ is a finite dimensional $G$-representation. We want to describe $X^{ss}$ and $X^{s}$ by describing their complements, $X^{us}$ and $X^{ns}$, respectively.

Let $\lambda \in \Hom(\Bbbk^{*}, G)$ be a one-parameter subgroup. $V$ has an induced $\Bbbk^{*}$-representation structure. Since $\Bbbk^{*}$ is abelian, we may find a basis $\{s_{0}, s_{1}, \ldots, s_{n}\}$ of $V$ and integers $w_0, \ldots, w_n$ such that 
\[
	\lambda(t) \cdot s_{i} = t^{w_{i}}s_{i}.
\]
\begin{definition}\label{def:mu}
Let $x \in \PP V^{*}$ and $\lambda$ be a one-parameter subgroup. We define a numerical function $\mu(x, \lambda)$ as 
\[
	\mu(x, \lambda) := \min \{w_{i}\;|\; s_{i}(x) \ne 0\}.
\]
\end{definition}

\begin{theorem}[\protect{Hilbert-Mumford criterion \cite[Theorem 2.1]{MFK94}, \cite[Theorem 9.1]{Dol03}}]\label{thm:HilbertMumford}
Let $G$ be a reductive group, $V$ be a finite dimensional $G$-representation and $x \in \PP V^{*}$. Then
\begin{enumerate}
\item $x$ is semi-stable if and only if $\mu(x, \lambda) \le 0$ for all $\lambda$;
\item $x$ is stable if and only if $\mu(x, \lambda) < 0$ for all $\lambda$. 
\end{enumerate}
\end{theorem}

To simplify the calculation, we may use the following `reduction-to-maximal-torus' trick (or the \emph{torus trick}, for simplicity). Observe that:
\begin{enumerate}
\item A point $x \in \PP V^{*}$ is (semi/poly-)stable if and only if $gx$ is (semi/poly-)stable for $g \in G$;
\item For any $x \in \PP V^{*}$, $\lambda \in \Hom(\Bbbk^{*}, G)$, and $g \in G$, we have $\mu(x, \lambda) = \mu(gx, g\lambda g^{-1})$.
\end{enumerate}

The image of any one-parameter subgroup is contained in a maximal torus of $G$. Furthermore, any two maximal tori of $G$ are conjugate to each other. So the Hilbert-Mumford criterion can be restated as:
\begin{theorem}[\protect{Hilbert-Mumford criterion, second version}]\label{thm:HilbertMumford2}
Let $G$ be a reductive group, $V$ be a finite dimensional $G$-representation and $x \in \PP V^{*}$. Then $x$ is $G$-(semi-)stable if and only if $x$ is $T$-semi-stable for all maximal tori $T$.
\end{theorem}

Thus, we may analyze (semi-)stability in two steps. 
\begin{enumerate}
\item Fix a maximal torus $T$ of $G$ and study (semi-)stability with respect to $T$;
\item Describe the $G$-orbit of each stratum of the unstable/non-stable locus with respect to $T$. In many cases, this step is done by describing each orbit geometrically or in a coordinate-free way. 
\end{enumerate}

The first step is a highly non-trivial combinatorial calculation, and we provide an algorithm for it in this paper, together with an implementation for simple groups of type $A$, $B$, $C$ and $D$ (the most common ones in applications) in \texttt{SageMath} \cite{sagemath-code}. Our algorithm works for all reductive groups. However, it is more efficient for semisimple groups. Representations of these groups have Weyl group symmetry, which we exploit to increase the efficiency of the algorithm  significantly. 

For many moduli problems, which are central applications of GIT calculation, the second step involves the geometry of parameterized objects. We do not focus on this step in this paper, but see Section \ref{sec:cubic} for an example. 

\subsection{State polytopes}

The Hilbert-Mumford criterion and the torus trick enable us to interpret (semi-)stability in terms of polyhedral geometry. The purpose of this section is to explain this connection. 

Let $G$ be a reductive group and let $T$ be a fixed maximal torus of $G$. Let $N := \Hom(\Bbbk^{*}, T)$ be the set of one-parameter subgroups, which has a lattice structure. Set $N_{\QQ} := N \otimes_{\ZZ}\QQ$ and $N_{\RR} := N\otimes_{\ZZ}\RR$. Then $N_{\QQ}$ and $N_{\RR}$ are finite dimensional vector spaces and their dimension is called the rank of $G$. 

Let $M := \Hom(T, \Bbbk^{*})$ be the group of characters. We may define $M_{\QQ}$ and $M_{\RR}$ in the same way. An element of $M_{\RR}$ is called a \emph{weight}. There is a perfect pairing 
\begin{eqnarray*}
	N \times M & \to & \ZZ\\
	(\lambda , \chi) & \mapsto & \langle \lambda, \chi\rangle := m, \mbox{ where } (\chi \circ \lambda) (t) = t^{m}.
\end{eqnarray*}

For a finite dimensional $G$-representation $V$, consider the induced $T$-action on $V$. Then there is a unique decomposition of $V$ as a direct sum of eigenspaces 
\[
	V = \bigoplus_{\chi \in M}V_{\chi}
\]
where $V_{\chi} = \{v \in V\;|\; \lambda(t) \cdot v = t^{\langle \lambda, \chi\rangle}v \text{ for all }\lambda \in N\}$. 

\begin{definition}\label{def:state}
The \emph{state} of $V$ is $\Xi_{V} := \{\chi \in M\;|\; V_{\chi} \ne 0\} \subset M$. For any $x \in \PP V^{*}$, the \emph{state} of $x$ is 
\[
	\Xi_{x} := \{\chi \in \Xi_{V}\;|\; \exists s \in V_{\chi},\, s(x) \ne 0\} \subset \Xi_{V}.
\]
\end{definition}
Note that the above definition depends on the choice of maximal torus $T$ (but not of the choice of basis for a given $T$). Since our second step above ultimately aims to describe the $G$-orbits of unstable/non-stable points in a coordinate-free way, this is inconsequential for the whole program. Note further that for a general $x \in \PP V^{*}$, $\Xi_{x} = \Xi_{V}$. 

\begin{remark}\label{rem:basicoperation}
Let $V$ and $W$ be two finite dimensional $T$-representations. It is straightforward to verify that $\Xi_{V \oplus W} = \Xi_{V} \cup \Xi_{W}$ and $\Xi_{V\otimes W}$ is the Minkowski sum of $\Xi_{V}$ and $\Xi_{W}$. The state of the wedge product $V \wedge W$ corresponds to a `truncation' of $\Xi_{V \otimes W}$. 
\end{remark}

Any nontrivial $\lambda \in N_{\RR}$ defines a linear functional $\ell_{\lambda} : M_{\RR} \to \RR$ by the formula  $\ell_{\lambda}(\chi)=\langle \lambda, \chi\rangle$.

\begin{definition}\label{def:unstablestate}
Fix a finite dimensional $T$-representation $V$. Let $\lambda \in N_{\RR}$. For any $c \in \RR$, we may define 
\[
	\Xi_{V, \lambda \ge c} := \{\chi \in \Xi_{V}\;|\; \langle \lambda, \chi \rangle \ge c\}.
\]
Similarly, 
\[
	\Xi_{V, \lambda > c} := \{\chi \in \Xi_{V}\;|\; \langle \lambda, \chi \rangle > c\}.
\]
We define $\Xi_{V, \lambda = c} := \Xi_{V, \lambda \ge c} \setminus \Xi_{V, \lambda > c}$. 
\end{definition}

We may restate the Hilbert-Mumford criterion (Theorem \ref{thm:HilbertMumford}) for a torus $T$, in terms of states. For a set $S \subset M_{\RR}$, the convex hull of $S$ is denoted by $\Conv(S)$. 

\begin{theorem}[\protect{Hilbert-Mumford criterion, third version \cite[Theorem 9.2]{Dol03}}]\label{thm:HilbertMumfordII}
\begin{enumerate}
\item A point $x \in \PP V^{*}$ is semi-stable with respect to $T$ if and only if $\Conv(\Xi_{x})$ contains the trivial character.
\item A point $x \in \PP V^{*}$ is stable with respect to $T$ if and only if the interior of $\Conv(\Xi_{x})$ contains the trivial character. 
\end{enumerate}
\end{theorem}

By the perfect pairing between $N$ and $M$, each one-parameter subgroup $\lambda$ induces a hyperplane in $M$ with the sign of $\langle \lambda, \xi\rangle$ being positive, zero or negative, depending whether the character $\xi$ is `over', `on' or `under' the hyperplane induced by $\lambda$, respectively. Thus, theorems \ref{thm:HilbertMumford} and \ref{thm:HilbertMumfordII} imply the following result.

\begin{corollary}\label{cor:unstability}
\begin{enumerate}
\item A point $x \in \PP V^{*}$ is unstable with respect to $T$ if and only if $\Xi_{x} \subset \Xi_{V, \lambda > 0}$ for some $\lambda \in N$. 
\item A point $x \in \PP V^{*}$ is non-stable with respect to $T$ if and only if $\Xi_{x} \subset \Xi_{V, \lambda \ge 0}$ for some $\lambda \in N$. 
\end{enumerate}
\end{corollary}

\begin{remark}\label{rmk:Liealg}
In addition to Corollary \ref{cor:unstability}, we can observe that $T$-stability is determined by the set of weights of the given $G$-representation $V$. In other words, it can be described by the associated Lie algebra $\mathfrak{g}$. Thus, any finite extension and finite quotient of an algebraic group induce the same (semi-)stable locus. For example, if one has a $\mathrm{PGL}_n$-action, one may replace it by a compatible $\SL_n$-action.
\end{remark}

Since $\Xi_{V}$ is a finite set of points, it is sufficient to check (semi)-stability with respect to  finitely many one-parameter subgroups. This explains the finiteness statement in Theorem \ref{thm:mainthm} even for arbitrary reductive group actions.

\begin{corollary}\label{cor:finite}
There is a finite set $\Lambda_{ss} := \{\lambda_{i}\}_{i \in I}$ of one-parameter subgroups such that $x \in \PP V^{*}$ is unstable with respect to $T$ if and only if $\Xi_{x} \subset \Xi_{V, \lambda_{i} > 0}$ for some $i \in I$. Equivalently, there is a finite set $P_{ss} := \{\Xi_{V, \lambda_{i} > 0}\}_{i \in I}$ of maximal unstable states. And there is a finite set $\Lambda_{s} := \{\lambda_{j}\}_{j \in J}$ of one-parameter subgroups such that $x \in \PP V^{*}$ is non-stable with respect to $T$ if and only if $\Xi_{x} \subset \Xi_{V, \lambda_{j} \ge 0}$ for some $j \in J$. Equivalently, there is a finite set $P_{s} := \{\Xi_{V, \lambda_{i} > 0}\}_{i \in I}$ of maximal  non-stable states. 
\end{corollary}

Therefore for the computation of the (semi-)stable locus, the following question is the first step. 

\begin{question}\label{que:algorithm}
Find \emph{efficient} algorithms to determine $P_{ss}$ and $P_{s}$. 
\end{question}

In the next section we present algorithms to find these two sets of one-parameter subgroups.
For any finite dimensional representation $V$ of a semisimple group $G$, $\Xi_{V}$ has a Weyl group symmetry around the origin. This group action is quite rich, allowing us to  reduce the size of the problem significantly. 

Recall that the \emph{Weyl group} $W$ of a semisimple group $G$ is defined as $W = N_{G}(T)/T$ where $N_{G}(T)$ is the normalizer of $T$ in $G$. Alternatively, $W$ may be encoded in the root datum of $G$. $W$ acts linearly on both $N$ and $M$, and induces an action on $\Xi_{V}$. The Weyl group action on any $G$-representation is induced from the action of $G$, in other words, the `coordinate change' by the group $G$. In particular, $W$ acts on the set of characters of any $G$-representation $V$ as reflections, so it is linearly extended to the actions on $M_{\RR}$ and $N_{\RR}$. 
Choose a general hyperplane on $M_\RR$ and take the set $R_+$ of roots on the one half of the hyperplane (they are called positive roots). There is a unique basis $\Delta \subset R_+$ of $M_\RR$ such that any positive root can be written as a nonnegative $\ZZ$-linear combination of $\Delta$ \cite[Section 10.1]{Humphreys}. This basis $\Delta$ is called a \emph{base}. Then the \emph{fundamental chamber} is the intersection of half-planes $\bigcap_{\alpha \in \Delta}\langle -, \alpha\rangle \ge 0$ \cite[Section 10.1]{Humphreys}. Note that a choice of a fundamental chamber depends on a choice of a base. Then $W$ acts transitively on the set of fundamental chambers \cite[Corollary D.32]{FH91}. 

Therefore, if we denote one fundamental chamber of the $W$-action on $N_{\RR}$ by $F$, then it is sufficient to find the maximal elements of the set $\{\Xi_{V, \lambda \ge 0}\}$ such that $\lambda \in F$, because other maximal elements will be obtained by applying the Weyl group symmetry. Thus, Question \ref{que:algorithm} is reduced to the following.

\begin{question}\label{que:algorithmrefined}
Fix a fundamental chamber $F$ of $N_{\RR}$. Let $P_{s}^{F}$ be the set of maximal elements in $\{\Xi_{V, \lambda \ge 0}\}$ such that $\lambda \in F$, and $P_{ss}^{F}$ be the set of maximal elements in $\{\Xi_{V, \lambda > 0}\}$ such that $\lambda \in F$. Find effective algorithms to calculate $P_{s}^{F}$ and $P_{ss}^{F}$.
\end{question}

Since a base in $M_\RR$ is a basis of $M_\RR$, the following fact is immediate from the definition
$$F = \bigcap_{\alpha \in \Delta}\langle -, \alpha\rangle \ge 0.$$

\begin{lemma}%
\label{lemma:semisimple-F}
The fundamental chamber $F$ is a simplicial cone. Thus, any vector in $F$ can be written uniquely as a non-negative linear combination of its ray generators.
\end{lemma}

\begin{remark}\label{rmk:reductiveexpensive}
If the group $G$ is not semisimple, then the state of $V$ need not have any symmetry, and then the computation can be significantly more expensive.  For instance, if $G = T$, any finite set $\Xi \subset M$ can be $\Xi_V$ for some representation $V$. So we cannot expect any symmetry on $\Xi_V$, and we need to compute the full sets $P_s$ and $P_{ss}$.
\end{remark}

\section{Algorithms}\label{sec:algorithm}

In this section, we describe algorithms to calculate two finite sets $P_{ss}^{F}$ and $P_{s}^{F}$ of maximal unstable states and of maximal non-stable states, described in Question \ref{que:algorithmrefined}.

Let $T\leqslant G$ be a choice of a maximal torus in a semisimple group $G$, and let $V$ be a finite dimensional $G$-representation. As before, $N$ is the lattice of one parameter subgroups of $T$, and $M$ is the lattice of characters, and $d = \mathrm{rank} \;N = \mathrm{rank}\; M$. If we denote by $\mathfrak{g}$ the Lie algebra associated to $G$, then $M$ is naturally identified with the weight lattice of $\mathfrak{g}$. Let $F$ be a fixed fundamental chamber in $N_{\RR}$  with respect to the Weyl group action. 

For the GIT quotient $\PP V^{*}\git G$, we need to calculate two finite sets $P_{ss}^{F}$ and $P_{s}^{F}$. The input of the algorithm is $\Xi_{V}
$, the set of characters of $V$ in $M$. The state $\Xi_{V}$ can be calculated using standard formulae in representation theory. For instance, in \texttt{SageMath} \cite{sagemath}, the `Weyl Character Ring' package can compute $\Xi_{V}$.

\subsection{Stable locus}\label{ssec:stablilityalgorithm}

A simple but important observation is that for any maximal $\Xi_{V, \lambda \ge 0}$, the set $\Xi_{V, \lambda = 0}$ must have at least $(d-1)$ linearly independent characters. Otherwise, by perturbing $\lambda$ by $\lambda'$, we would be able to obtain a strictly larger state $\Xi_{V, \lambda' \ge 0}$. Therefore, we have the following outline of an  algorithm. 

\begin{enumerate}
\item Let $\cC$ be the set of all $(d-1)$ linearly independent subsets of characters
\[
    \{\chi_{1}, \chi_{2}, \ldots, \chi_{d-1}\} \subset \Xi_{V}.
\]
\item For each subset $I \in \cC$, compute a nontrivial $\lambda \in N$ such that $\langle \lambda, \chi_{i}\rangle = 0$ for all $\chi_{i} \in I$. By the linearly independence of $I$, up to a scalar multiple, $\lambda$ is unique.  Let $\Lambda$ be the set of such $\lambda$'s which is in $F$. 
\item For each $\lambda \in \Lambda$, compute $\Xi_{V, \lambda \ge 0}$. Let $\cS$ be the set of such $\Xi_{V, \lambda \ge 0}$'s. 
\item Let $\cS_{m} \subset \cS$ be the set of maximal elements with respect to the inclusion order. Then $P_{s}^{F} \subset \cS_{m}$.
\end{enumerate}

\begin{remark}
Note that $\cS_{m}$ is the set of maximal elements in $\{\Xi_{V, \lambda \ge 0}\}_{\lambda \in F}$, while $P_{s}^{F}$ is the set of maximal elements in $\{\Xi_{V, \lambda \ge 0}\}$ such that $\lambda \in F$. Clearly $P_{s}^{F} \subset \cS_{m}$. We can compute $P_{s}^{F}$ from $\cS_{m}$ effectively. We describe an algorithm for this later. 
\end{remark}

In general the algorithm outlined above will be slow, because the set $\cC$ is very large. However, we can improve its performance by reducing the number of characters that need to be considered. We do not need to consider the whole set $\Xi_{V}$ of characters to calculate the set of subsets; instead, it is enough to use a proper subset $\Xi_{V}^{E, s} \subset \Xi_{V}$ that we call the \emph{essential} characters for the stability calculation. We give more details below.

Each character $\chi \in \Xi_{V} \setminus \{0\}$ defines a hyperplane $H_{\chi} := \{\lambda \in N_{\RR}\;|\; \langle \lambda, \chi \rangle = 0\}$ on $N_{\RR}$. Suppose that $H_{\chi} \cap F = \{0\}$. Then for every $(d-1)$ subset of characters $I$ which contains $\chi$, the one-parameter subgroup $\lambda$ that $I$ determines is not on $F$ (because it is on $H_{\chi}$). Therefore, we may discard such $\chi$. 

\begin{proposition}\label{prop:differenceforstability}
Let $\gamma_{1}, \gamma_{2}, \ldots, \gamma_{d}$ be ray generators of $F$. For $\chi \in \Xi_{V} \setminus \{0\}$, $H_{\chi} \cap F \ne \{0\}$ if and only if
\begin{equation}\label{eqn:differenceforstability}
	\chi \in \bigcup_{i=1}^{d}\Xi_{V, \gamma_{i} \ge 0} \setminus \bigcap_{i=1}^{d}\Xi_{V, \gamma_{i} > 0}.
\end{equation}
\end{proposition}

\begin{proof}
If $\chi \in \bigcap_{i=1}^{d}\Xi_{V, \gamma_{i} > 0}$, then $\langle \gamma_{i}, \chi \rangle > 0$ for all $i$. Since by Lemma \ref{lemma:semisimple-F}, any $\lambda \in F$ can be written uniquely as a non-negative linear combination of $\{\gamma_{i}\}$, we have that $\langle \lambda, \chi \rangle > 0$ for all $\lambda \in F \setminus \{0\}$. Therefore $H_{\chi} \cap F = \{0\}$. If $\chi \notin \bigcup_{i=1}^{d}\Xi_{V, \gamma_{i} \ge 0}$, then $\langle \gamma_{i}, \chi \rangle < 0$ for all $i$, and hence $\langle \lambda, \chi \rangle < 0$ for all $\lambda \in F \setminus \{0\}$. Thus $H_{\chi} \cap F = \{0\}$. Thus $H_{\chi} \cap F \ne \{0\}$ implies \eqref{eqn:differenceforstability}.

Conversely, if \eqref{eqn:differenceforstability} holds, then there is one $\gamma_{i}$ such that $\langle \gamma_{i}, \chi\rangle \ge 0$ and there is one $\gamma_{j}$ such that $\langle \gamma_{j}, \chi \rangle \le 0$. By taking a nontrivial positive linear combination of $\gamma_{i}$ and $\gamma_{j}$, we may find $\lambda \in F \setminus \{0\}$ such that $\langle \lambda, \chi \rangle = 0$. Then $\lambda \in H_{\chi} \cap F$. 
\end{proof}

Another observation is that if $\chi_{1}, \chi_{2} \in \Xi_{V}$ are proportional to each other, then we may discard one of them. Indeed, for any subset $J$ of size $d-2$, $I_{1} := J \cup \{\chi_{1}\}$ and $I_{2} := J \cup \{\chi_{2}\}$ define the same $\lambda \in N_{\RR}$, hence the same $\Xi_{V, \lambda \ge 0}$. Thus, we define:

\begin{definition}
    A \emph{set of essential characters} $\Xi_{V}^{E,s}$ is a maximal subset of the right hand side of \eqref{eqn:differenceforstability} where no two elements are proportional to each other.
\end{definition}
Note that $\Xi_{V}^{E,s}$ is not uniquely defined, since any character in $\Xi_{V}^{E,s}$ can be replaced by a proportional one to it and it will still satisfy the definition. However, for our purposes, this will not make a difference.

Finally, we explain how to compute $P_{s}^{F}$ from $\cS_{m}$. 

\begin{definition}\label{def:Wprime}
Let $W$ be the Weyl group of $G$ and let $\Xi_{V, \lambda \ge 0} \in \cS_{m}$. Let $W' \subset W$ be the set of all non-identity elements that move the fundamental chamber $F$ to  another cone $F'$ that intersects $F$ non-trivially, so $F \cap F' \ne \{0\}$. 
\end{definition}

\begin{remark}\label{rem:Wprime}
The only element in $W$ which preserves $F$ is the identity. For any two chambers $F$ and $F'$, there is a unique element in $W$ which maps $F$ to $F'$. Now if $F'$ intersects $F$ nontrivially, we can make a sequence of reflections that maps $F$ to $F'$ while fixing the intersection. Thus, $W' = \bigcup_{i} \mathrm{Stab}(\gamma_i) \setminus \{e\}$, where $\gamma_i$ are the generators of $F$.
\end{remark}

\begin{lemma}\label{lem:refiningmaximalset}
Let $\Xi_{V, \lambda \ge 0} \in \cS_m$. $\Xi_{V, \lambda \ge 0} \in P_{s}^{F}$ if and only if there is no $g \in W'$ and $\Xi_{V, \mu \ge 0} \in \cS_{m}$ such that $\Xi_{V, g\lambda \ge 0} \subsetneq \Xi_{V, \mu \ge 0}$.
\end{lemma}

\begin{proof}
Suppose that $\Xi_{V, \lambda \ge 0} \in P_{s}^{F}$. Since it is maximal in $\{\Xi_{V, \nu \ge 0}\}$ for all $\nu \in N$,
$\Xi_{V, \lambda \ge 0} \supset \Xi_{V, g\mu \ge 0}$ for all $g \in W'$ and $\Xi_{V, \mu \ge 0} \in \cS_{m}$ such that $g\mu\in F$. But $\Xi_{V, \lambda \ge 0} \supset \Xi_{V, g\mu \ge 0}$ is equivalent to $\Xi_{V, g^{-1}\lambda \ge 0} \supset \Xi_{V, \mu \ge 0}$. By Remark \ref{rem:Wprime}, $g^{-1} \in W'$. 

Conversely, suppose that $\Xi_{V, \lambda \ge 0} \notin P_{s}^{F}$. This implies that $\Xi_{V, \lambda \ge 0}$ is not maximal in $\{\Xi_{V, \nu \ge 0}\}$. This is possible if there is $\mu \in N$ so that $\Xi_{V, \lambda \ge 0} \subsetneq \Xi_{V, \mu \ge 0}$ where $\mu$ is not in $F$, but its adjacent cone $F'$. We may assume that $\Xi_{V, \mu \ge 0}$ is maximal. Since $F$ is a fundamental chamber, there is $g \in W'$ such that $g\mu \in F$. So $\Xi_{V, g\mu \ge 0} \in P_{s}^{F} \subset \cS_{m}$. Therefore $\Xi_{V, g\lambda \ge 0} \subsetneq \Xi_{V, g\mu \ge 0} \in \cS_{m}$. 
\end{proof}

By combining these ideas, we can make an optimized algorithm for the computation of $P_{s}^{F}$:

\begin{algorithm}\label{alg:Stable}
[Algorithm for the computation of $P_{s}^{F}$]
\leavevmode\\
\textbf{Input:} The state $\Xi_{V}$.\\
\textbf{Output:} The set of maximal non-stable states $P_{s}^{F}$.\\
\begin{enumerate}[label=\arabic*.]
\item $A_{0} := \Xi_{V}$
\item $A_{1} := \bigcup_{i=1}^{d}\Xi_{V, \gamma_{i} \ge 0} \setminus \bigcap_{i=1}^{d}\Xi_{V, \gamma_{i} > 0}$
\item $A_{2} := A_{1} \setminus \{0\}$
\item $A_{3} := \emptyset$
\item \textbf{for all} $\chi \in A_{2}$ \textbf{do}
\item \qquad \texttt{is\_dependent} := \texttt{false}
\item \qquad \textbf{for all} $\chi' \in A_{3}$
\item \qquad \qquad \textbf{if} $\chi = c\chi'$ \textbf{for some} $c \in \RR$ \textbf{then} \texttt{is\_dependent} := \texttt{true}
\item \qquad \textbf{if} \texttt{is\_dependent} = \texttt{false} \textbf{then} $A_{3} := A_{3} \cup \{\chi\}$
\item $\cS_{m} := \emptyset$
\item \textbf{for all} $I \in \binom{A_{3}}{d-1}$ \textbf{do}
\item \qquad \textbf{if} $I$ is linearly independent \textbf{then do}
\item \qquad \qquad Calculate $\lambda \ne 0$ such that $\langle \lambda, \chi\rangle = 0$ for all $\chi \in I$
\item \qquad \qquad \textbf{if} $\lambda \notin F$ \textbf{then} $\lambda := -\lambda$
\item \qquad \qquad \textbf{if} $\lambda \in F$ \textbf{then do}
\item \qquad \qquad \qquad Compute $\Xi_{V, \lambda \ge 0}$
\item \qquad \qquad \qquad \texttt{is\_maximal} := \texttt{true}
\item \qquad \qquad \qquad \textbf{for all} $\Xi_{V, \mu \ge 0} \in \cS_{m}$ \textbf{do}
\item \qquad \qquad \qquad \qquad \textbf{if} $\Xi_{V, \lambda \ge 0} \subset \Xi_{V, \mu \ge 0}$ \textbf{then} \texttt{is\_maximal} := \texttt{false} and \textbf{break}
\item \qquad \qquad \qquad \qquad \textbf{if} $\Xi_{V, \lambda \ge 0} \supset \Xi_{V, \mu \ge 0}$ \textbf{then} $\cS_{m} := \cS_{m} \setminus \{\Xi_{V, \mu \ge 0}\}$
\item \qquad \qquad \qquad \textbf{if} \texttt{is\_maximal} = \texttt{true} \textbf{then} $\cS_{m} := \cS_{m} \cup \{\Xi_{V, \lambda \ge 0}\}$
\item $P_{s}^{F} := \emptyset$
\item \textbf{for all} $\Xi_{V, \lambda \ge 0} \in \cS_{m}$ \textbf{do}
\item \qquad \texttt{is\_maximal} := \texttt{true}
\item \qquad \textbf{for all} $g \in W'$ \textbf{do}
\item \qquad \qquad \textbf{for all} $\Xi_{V, \mu \ge 0} \in \cS_{m}$ \textbf{do}
\item \qquad \qquad \qquad \textbf{if} $\Xi_{V, g\lambda \ge 0} \subsetneq \Xi_{V, \mu \ge 0}$ \textbf{then} \texttt{is\_maximal} := \texttt{false} and \textbf{break}
\item \qquad \textbf{if} \texttt{is\_maximal} = \texttt{true} \textbf{then} $P_{s}^{F} := P_{s}^{F} \cup \{\Xi_{V, \lambda \ge 0}\}$
\item \textbf{return} $P_{s}^{F}$
\end{enumerate}
\end{algorithm}

\begin{remark}\label{rem:reductivestable}
Algorithm \ref{alg:Stable} works for a more general reductive group, after the following modification. If the group is not semisimple, we may not expect any symmetry on $\Xi_V$, so we need to set $A_1 = A_2 = \Xi_V$ and let $F = N_\RR$. The rest of the algorithm works the same. 
\end{remark}

\subsection{Semistable locus}\label{ssec:semistabilityalgorithm}

In this section we present an algorithm to calculate the semi-stable locus. This algorithm is a generalization of the algorithm described in \cite{GMG18} (cf. \cite{GMG19, GMGZ18}), which considers a special case of $G = \SL_{r}$ and $V = \Sym^{d}\CC^{r}\otimes \Sym^{e}\CC^{r}$. 

In this section, we assume that $G$ is a semisimple group of rank $d$, and $T$ be a fixed maximal torus of $G$. 

For the semi-stable locus computation, we need one technical assumption. 

\begin{assumption}\label{ass:nonempty}
From now on, we assume that the $T$-\emph{stable} locus $\PP V^{* s}(T)$ is nonempty. Equivalently, we assume that the state $\Xi_V$ is full-dimensional and the trivial character $\chi_0$ is in $\intr\Conv(\Xi_V)$.
\end{assumption}

Assumption \ref{ass:nonempty} is true for many GIT problems, as we illustrate in the following lemmas.

\begin{lemma}\label{lem:assumptionsimplecase}
Let $G$ be a simple group. For any nontrivial $G$-representation $V$, Assumption \ref{ass:nonempty} holds. 
\end{lemma}

\begin{proof}
By Theorem \ref{thm:HilbertMumfordII}, it is sufficient to show that the trivial character $\chi_0$ is in the interior of $\Conv(\Xi_V)$. 

First, assume that $G$ is a simple group and $V$ is an irreducible representation. Then $\Xi_V$ has a nontrivial character. Since the Weyl group $W$ is generated by reflections associated to roots, and the roots span $N_\RR$, $\Conv(\Xi_V)$ is top-dimensional $W$-invariant polytope. Choose any $\chi \in \Conv(\Xi_V)$ and set $\tau = \frac{1}{|W|}\sum_{g \in W} g\chi$. Then $\tau \in \Conv(\Xi_V)$ and $W$-invariant. Since the only $W$-invariant vector is zero, $\tau = \chi_0$. Note that $|W| > d$. Since $\chi_0$ is a positive linear combination of linearly dependent vectors in $\Conv(\Xi_V)$, $\chi_0$ is in the interior of $\Conv(\Xi_V)$.

When $V$ is not irreducible, let $W$ be a nontrivial irreducible factor of $V$. Then $\Xi_{W} \subset \Xi_{V}$. So $\Xi_{V}$ contains the trivial character in its interior, too.
\end{proof}

For a general semisimple group $G$, by Remark \ref{rmk:Liealg}, it is sufficient to consider the case that $G = G_1 \times \cdots \times G_k$, where $G_i$ are simple groups. 

\begin{lemma}\label{lem:assumptionproductcase}
Let $G = G_{1} \times \cdots \times G_{k}$ be a semisimple group which is a product of simple groups.
Let $V$ be a finite dimensional $G$-representation whose induced $G_{i}$-representations are nontrivial for all $1 \le i \le k$. Then $V$ satisfies Assumption \ref{ass:nonempty}.
\end{lemma}

\begin{proof}
Here we give a proof when $G = G_{1} \times G_{2}$. The general case easily follows by induction. Since the two induced representations are nontrivial, there are two possibilities: $V$ has an irreducible factor of the form $V_{1}\otimes V_{2}$, where $V_{i}$ is a nontrivial irreducible $G_{i}$-representation, or $V$ has two irreducible factors $V_{1}$ and $V_{2}$ where $V_{i}$ is an irreducible $G_{i}$-representation and $G_{2-i}$ acts trivially. Note that $M_{\RR} \cong M_{1 \RR} \oplus M_{2 \RR}$ where $M_{i \RR}$ is the space of characters of $G_{i}$, and $W \cong W_{1} \times W_{2}$ where $W_{i}$ is the Weyl group of $G_{i}$. In each case, $\Conv(\Xi_{V})$ is top-dimensional. Therefore one can argue it in the exactly same way to the proof of Lemma \ref{lem:assumptionsimplecase}: The average of an element of $\Conv(\Xi_{V_{1}\otimes V_{2}})$ or $\Conv(\Xi_{V_{1}\oplus V_{2}})$ is $W$-invariant, thus it is a trivial character $\chi_{0}$ and it is in the interior of $\Conv(\Xi_{V})$. 
\end{proof}

\begin{remark}
That Assumption \ref{ass:nonempty} holds for a fixed maximal torus $T$ (even for all maximal tori!) does not imply the nonemptiness of the $G$-stable locus $\PP V^{*s}(G)$  when $\dim V$ is small compared to $\dim G$. For instance, consider $G = \SL_3$ and $V = \sym^2 \CC^3$. Then by Lemma \ref{lem:assumptionsimplecase}, for every maximal torus $T \subset \SL_3$, the $T$-stable locus $\PP V^{* s}(T)$ is nonempty. However, $\dim V = 6 < \dim \SL_3 = 8$, so every point on $\PP V^{*ss}$ has a positive dimensional stabilizer, hence it is non-stable. 
\end{remark}

\begin{lemma}\label{lem:maxdim}
Suppose Assumption \ref{ass:nonempty} is true. Then for any maximal unstable state $\Xi_{V, \lambda > 0}$, $\dim \Conv(\Xi_{V, \lambda > 0})\ge d-1$. Furthermore, if $\dim \Conv(\Xi_{V, \lambda > 0}) = d-1$, then $\Xi_{V, \lambda > 0}$ is not contained in a hyperplane passing through the origin. 
\end{lemma}

\begin{proof}
Suppose that $\dim \Conv(\Xi_{V, \lambda > 0}) < d-1$. Then there is a hyperplane $H \subset M_{\RR}$ such that $\Xi_{V, \lambda > 0} \subset H$. Since Assumption \ref{ass:nonempty} is true, $\dim \Conv(\Xi_{V}) = d$.
So there must be $\chi' \in \Xi_{V} \setminus H$. There is a linear functional $\ell : M_{\RR} \to \RR$ such that $\ell|_{H} = 0$ and $\ell(\chi') = 1$. Because there is a perfect pairing $N_{\RR} \times M_{\RR} \to \RR$, there is $\mu \in N_{\RR}$ such that $\ell(\chi) = \langle \mu, \chi\rangle$ for all $\chi \in M_{\RR}$. Now for $m \gg 0$, $\Xi_{V, \lambda + m\mu > 0} \supset \Xi_{V, \lambda} \cup \{\chi'\}$. It contradicts the maximality of $\Xi_{V, \lambda > 0}$, proving the first statement. 

Now suppose that $\dim \Conv(\Xi_{V, \lambda > 0}) = d - 1$. If there is a hyperplane $H \subset M_\RR$ passing through the origin and $\Xi_{V, \lambda > 0} \subset H$, then we can choose $\chi' \in \Xi_V \setminus H$ and $\ell : M_\RR \to \RR$ such that $\ell|_H = 0$ and $\ell (\chi') = 1$. Then we may argue in the same way to show the non-maximality of $\Xi_{V, \lambda > 0}$ as before. 
\end{proof}

The next proposition is the key observation for the computation of the semi-stable locus. 

\begin{proposition}\label{prop:minimumacheiveddtimes}
Suppose Assumption \ref{ass:nonempty} is true. Let $\Xi_{V, \lambda > 0}$ be a maximal unstable state. There is $\lambda ' \in N_{\RR}$ such that:
\begin{enumerate}
\item $\Xi_{V, \lambda > 0} = \Xi_{V, \lambda' > 0}$;
\item The minimum value of $\langle \lambda', \chi \rangle$ for all $\chi \in \Xi_{V, \lambda' > 0}$ is achieved at $d$ linearly independent characters $\chi_{1}, \chi_{2}, \ldots, \chi_{d} \in \Xi_{V, \lambda' > 0}$.
\end{enumerate}
\end{proposition}

\begin{proof}
By Lemma \ref{lem:maxdim}, we know that $\dim \Conv(\Xi_{V, \lambda > 0}) \ge d-1$.

First of all, suppose that $\dim \Conv(\Xi_{V, \lambda > 0}) = d$. Since it is a convex polytope, it is an intersection of finitely many half-spaces:
\[
	\Conv(\Xi_{V, \lambda > 0}) = \bigcap_{k \in K}H_{\lambda_{k} \ge c_{k}}
\]
where $H_{\lambda_{k} \ge c_{k}} := \{\chi \in M_{\RR}\;|\; \langle \lambda_{k}, \chi \rangle \ge c_{k}\}$ and $K$ is a finite index set. Furthermore, by eliminating redundant half-spaces, we may assume that for all $k\in K$, $H_{\lambda_{k} \ge c_{k}} \cap \Conv(\Xi_{V, \lambda > 0})$ is a ($(d-1)$-dimensional) facet of $\Conv(\Xi_{V, \lambda > 0})$, so $H_{\lambda_{k} \ge c_{k}}$ contains at least $d$ linearly independent characters in $\Xi_{V, \lambda > 0}$. Since $\Xi_{V, \lambda > 0}$ is unstable, the trivial character $\chi_{0}$ is not in $\Conv(\Xi_{V, \lambda > 0})$. Thus there is $k \in K$ such that $\chi_{0} \notin H_{\lambda_{k} \ge c_{k}}$. Note that this implies $c_k > 0$.

We claim that we may take $\lambda' = \lambda_{k}$. Clearly
$\Xi_{V, \lambda > 0} \subset \Xi_{V, \lambda_{k} \ge c_{k}} \subset \Xi_{V, \lambda_{k} > 0}$. By the maximality of $\Xi_{V,\lambda > 0}$, 
\[
	\Xi_{V, \lambda_{k} > 0} = \Xi_{V, \lambda_{k} \ge c_{k}} = \Xi_{V, \lambda > 0}. 
\]
From the first equality, we obtain the minimum value of $\langle \lambda', \chi \rangle = \langle \lambda_k, \chi \rangle$ is $c_{k}$. We checked that $H_{\lambda_{k} = c_{k}}$ has $d$ linearly independent characters of $\Xi_{V, \lambda' > 0}$.

Now suppose that $\dim \Conv(\Xi_{V, \lambda > 0}) = d-1$, so $\Conv(\Xi_{V, \lambda > 0})$ is `thin'. Let $A$ be the unique hyperplane (not passing through the origin by Lemma \ref{lem:maxdim}) containing $\Xi_{V, \lambda > 0}$. Take $\ell \in M_{\RR}^{*}$ such that $\ell|_{A} = c$ for some $c > 0$. Find $\lambda' \in N_{\RR}$ such that $\ell(\chi) = \langle \lambda', \chi\rangle$. Then $\lambda'$ is what we want. Moreover, since $\dim \Conv(\Xi_{V, \lambda' > 0}) = d - 1$, it has at least $d$ linearly independent characters $\chi_{1}, \chi_{2}, \ldots, \chi_{d}$ which correspond to vertices of $\Conv(\Xi_{V, \lambda' > 0})$ and $\lambda'(\chi_{i}) \equiv c$.
\end{proof}

Proposition \ref{prop:minimumacheiveddtimes} suggests the following outline for an algorithm to describe all maximal unstable states.

\begin{enumerate}
\item Let $\cC$ be the set of all $d$ linearly independent subsets of characters $\{\chi_{1}, \chi_{2}, \ldots, \chi_{d}\} \subset \Xi_{V}$. 
\item For each subset $I \in \cC$, compute a one-parameter subgroup $\lambda \in N_{\RR}$ such that $\langle \lambda, \chi_{i}\rangle = \langle \lambda, \chi_{j}\rangle > 0$ for all $\chi_{i}, \chi_{j} \in I$. It is unique up to a positive scalar multiple. Let $\Lambda$ be the set of such $\lambda$'s which is in our previous choice of fundamental chamber $F$. 
\item For each $\lambda \in \Lambda$, compute $\Xi_{V, \lambda > 0}$. Let $\cS$ be the set of such $\Xi_{V, \lambda > 0}$'s. 
\item Let $\cS_{m} \subset \cS$ be the set of maximal elements with respect to the inclusion order. Then $P_{ss}^{F} \subset \cS_{m}$. 
\end{enumerate}

As in the case of the stable locus, it may be possible that $P_{ss}^{F}$ is a proper subset of $\cS_{m}$. We may calculate $P_{ss}^{F}$ from $\cS_{m}$ by using Lemma \ref{lem:refiningmaximalset} with an obvious modification:

\begin{lemma}\label{lem:refiningmaximalsetsemistable}
Let $\Xi_{V, \lambda > 0} \in \cS_m$. $\Xi_{V, \lambda > 0} \in P_{ss}^{F}$ if and only if there is no $g \in W'$ and $\Xi_{V, \mu > 0} \in \cS_{m}$ such that $\Xi_{V, g\lambda > 0} \subsetneq \Xi_{V, \mu > 0}$.
\end{lemma}

The computational bottleneck of this approach is the computation of the set $\cC$ as before. Here we again calculate a proper subset $\Xi_{V}^{E,ss} \subset \Xi_{V}$ of essential characters. 

\begin{lemma}\label{lem:optimization1}
Let $\gamma_{1}, \gamma_{2}, \ldots, \gamma_{d}$ be the ray generators of $F$. Let $I = \{\chi_{1}, \chi_{2}, \ldots, \chi_{d}\}$ be a linearly independent $d$-subset of characters in $\Xi_{V}$. Suppose that 
\[	
	\chi_{1} \notin \bigcup_{i=1}^{d}\Xi_{V, \gamma_{i} > 0}.
\]
Let $\lambda \in N_{\RR}$ such that $\langle \lambda, \chi_{i} \rangle = \langle \lambda, \chi_{j}\rangle > 0$ for all $\chi_{i}, \chi_{j} \in I$, which is unique up to a nonzero scalar multiple. Then $\lambda \notin F$. 
\end{lemma}

\begin{proof}
By the assumption on $\chi_{1}$, $\langle \gamma_{j}, \chi_{1}\rangle \le 0$ for all $j$. If $\lambda \in F$, then $\lambda$ is a nonnegative linear combination of $\gamma_{1}, \gamma_{2}, \ldots, \gamma_{d}$. Thus $\langle \lambda, \chi_{1}\rangle \le 0$, which contradicts one of the assumptions. Therefore $\lambda \notin F$. 
\end{proof}

Lemma \ref{lem:optimization1} tells us that to construct a linearly independent $d$-subset of characters that define $\lambda$ to form a $\Xi_{V, \lambda > 0}$ with $\lambda \in F$, it is sufficient to take the characters on 
\[
	\bigcup_{i=1}^{d}\Xi_{V, \gamma_{i} > 0}.
\]

Note that $\Xi_{V, \lambda > 0} = \Xi_{V, \lambda \ge c}$ for some $c > 0$. By perturbing $\lambda$ slightly, we may assume that the supporting affine hyperplane $\Xi_{V, \lambda = 0}$ has $d$-linearly independent characters. The $d$-subset of characters that we will choose lie on the supporting hyperplane of $\Conv(\Xi_{V, \lambda = c})$. Then we expect that if $\chi \in \Xi_{V}$ lies on $\Xi_{V, \lambda > c} = \Xi_{V, \lambda \ge c} \setminus \Xi_{V, \lambda = c}$ for every $\lambda \in F$, then we do not need to use $\chi$ to construct $d$-subsets. 

Let $K := \bigcap_{i=1}^{d}\Xi_{V, \gamma_{i} > 0}$. Define a partial order $>$ on $K$ as $\chi > \chi'$ if and only if $\langle \gamma_{i}, \chi\rangle > \langle \gamma_{i}, \chi' \rangle$ for all $i$. 

\begin{lemma}\label{lem:optimization2}
Let $K_{nm} \subset K = \bigcap_{i=1}^{d}\Xi_{V, \gamma_{i} > 0}$ be the set of non-minimal elements of $K$ with respect to $>$. Let $\chi \in K_{nm}$. Then for every maximal unstable state $\Xi_{V, \lambda > 0} = \Xi_{V, \lambda \ge c}$ with $\lambda \in F \setminus \{0\}$, $\chi \in \Xi_{V, \lambda > c}$. 
\end{lemma}

\begin{proof}
Being $\chi \in K_{nm}$ means that there is $\chi' \in K$ such that $\langle \gamma_{i}, \chi \rangle > \langle \gamma_{i}, \chi'\rangle$ for all $i$. Since $\lambda \in F \setminus \{0\}$, $\lambda$ can be written uniquely as a nontrivial  linear combination $\sum a_{i}\gamma_{i}$ with $a_{i} \ge 0$. Then
\[
	\langle \lambda, \chi \rangle = \sum a_{i}\langle \gamma_{i}, \chi\rangle > \sum a_{i}\langle \gamma_{i}, \chi'\rangle = \langle \lambda, \chi' \rangle > 0.
\]
Thus $\chi' \in \Xi_{V, \lambda > 0}$. Since $\Xi_{V, \lambda > 0} = \Xi_{V, \lambda \ge c}$, $\langle \lambda, \chi'\rangle \ge c$ and $\chi \in \Xi_{V, \lambda > c}$. 
\end{proof}

Therefore to construct the set of $d$-subsets of characters, it is sufficient to consider the set 
\begin{equation}\label{eqn:essentialsetforss}
	\Xi_V^{E, ss} := \bigcup_{i=1}^{d}\Xi_{V, \gamma_{i} > 0} \setminus K_{nm}.
\end{equation}

\begin{remark}
On the other hand, we cannot eliminate one of two proportional characters, as we can for the stable locus computation in Section \ref{ssec:stablilityalgorithm}. This is because the semi-stable locus computation is based on supporting \emph{affine} spaces, not hyperplanes passing through the origin. 
\end{remark}

Based on these observations, below is the optimized  algorithm. 

\begin{algorithm}\label{alg:Semistable}
[Algorithm for the computation of $P_{ss}^{F}$]\leavevmode\\
\textbf{Input:} The state $\Xi_{V}$.\\
\textbf{Output:} The set $P_{ss}^{F}$ of maximal unstable states.\\
\begin{enumerate}[label=\arabic*.]
\item $B_{0} := \Xi_{V}$.
\item $B_{1} := \bigcup_{i=1}^{d}\Xi_{V, \gamma_{i} > 0}$
\item $K := \bigcap_{i=1}^{d}\Xi_{V, \gamma_{i} > 0}$
\item $J := K$
\item \textbf{for all} $\chi \in J$ \textbf{do}
\item \qquad \texttt{is\_minimal} := \texttt{true}
\item \qquad $J_{> \chi} := \bigcap_{i=1}^{d}\{\chi' \in J\;|\; \langle \gamma_{i}, \chi'\rangle > \langle \gamma_{i}, \chi \rangle\}$.
\item \qquad \textbf{if} $J_{> \chi} \ne \emptyset$ \textbf{then do}
\item \qquad \qquad $J := J \setminus J_{> \chi}$
\item \qquad \qquad \texttt{is\_minimal} := \texttt{false}
\item \qquad \qquad \textbf{break}
\item \textbf{if} \texttt{is\_minimal} = \texttt{false} \textbf{go to} Step 5.
\item $K_{nm} = K \setminus J$
\item $B_{2} := B_{1} \setminus K_{nm}$
\item $\cS_{m} = \emptyset$
\item \textbf{for all} $I \in \binom{B_{2}}{d}$ \textbf{do}
\item \qquad \textbf{if} $I$ is linearly independent \textbf{then do}
\item \qquad \qquad Calculate $\lambda \ne 0$ such that $\langle \lambda, \chi_{i}\rangle = \langle \lambda, \chi_{j}\rangle > 0$ for all $\chi_{i}, \chi_{j} \in I$.
\item \qquad \qquad \textbf{if} $\lambda \in F$ \textbf{then do}
\item \qquad \qquad \qquad Compute $\Xi_{V, \lambda > 0}$.
\item \qquad \qquad \qquad \texttt{is\_maximal} := \texttt{true}
\item \qquad \qquad \qquad \textbf{for all} $\Xi_{V, \mu > 0} \in \cS_{m}$ \textbf{do}
\item \qquad \qquad \qquad \qquad \textbf{if} $\Xi_{V, \lambda > 0} \subset \Xi_{V, \mu > 0}$ \textbf{then} \texttt{is\_maximal} := \texttt{false} and \textbf{break}
\item \qquad \qquad \qquad \qquad \textbf{if} $\Xi_{V, \lambda > 0} \supset \Xi_{V, \mu > 0}$ \textbf{then} $\cS_{m} := \cS_{m} \setminus \{\Xi_{V, \mu > 0}\}$. 
\item \qquad \qquad \qquad \textbf{if} \texttt{is\_maximal} = \texttt{true} \textbf{then} $\cS_{m} := \cS_{m} \cup \{\Xi_{V, \lambda > 0}\}$.
\item $P_{ss}^{F} := \emptyset$
\item \textbf{for all} $\Xi_{V, \lambda > 0} \in \cS_{m}$ \textbf{do}
\item \qquad \texttt{is\_maximal} := \texttt{true}
\item \qquad \textbf{for all} $g \in W'$ \textbf{do}
\item \qquad \qquad \textbf{for all} $\Xi_{V, \mu > 0} \in \cS_{m}$ \textbf{do}
\item \qquad \qquad \qquad \textbf{if} $\Xi_{V, g\lambda > 0} \subsetneq \Xi_{V, \mu > 0}$ \textbf{then} \texttt{is\_maximal} := \texttt{false} and \textbf{break}
\item \qquad \textbf{if} \texttt{is\_maximal} = \texttt{true} \textbf{then} $P_{ss}^{F} := P_{ss}^{F} \cup \{\Xi_{V, \lambda > 0}\}$
\item \textbf{return} $P_{ss}^{F}$
\end{enumerate}
\end{algorithm}

\begin{remark}\label{rem:reductivesemistable}
Suppose that $G$ is reductive but not semisimple, or Assumption \ref{ass:nonempty} is not satisfied. One may not expect a similar algorithm in general, because the state polytope $\Xi_V$ can be contained in an affine space of  large codimension, so there is no easy way to describe a maximal unstable state with $d$-set of linearly independent characters. This problem can be resolved if we replace $M_\RR$ by the smallest linear subspace of $M_\RR$ that contains $\Conv(\Xi_V)$. 

Even after that, since we do not have any Weyl group symmetry, we need to consider all $d$-sets of characters. So if we set $B_2 = B_1 = \Xi_V$ and set $F = N_\RR$, the algorithm gives the correct output. 
\end{remark}

\subsection{GIT boundary}\label{ssec:GITboundary}

When one studies the geometry of moduli spaces constructed by GIT, it is essential to study the geometry of the strictly polystable locus.
It enables us to apply Kirwan's partial desingularization procedure \cite{Kir85} to obtain a moduli space with better singularities, or to apply the  wall-crossing analysis as the linearization varies \cite{DH98, Tha96}. In this section, we describe the $G$-polystable locus and an algorithm to find $T$-polystable loci, where $T$ is a maximal torus of $G$.

Set theoretically, the image of the strictly semistable locus in the quotient, namely the points in $(X\git_L G) \setminus (X^s/G)$, is not in a bijection with the set of $G$-orbits of strictly semistable points, but that of polystable points. Recall that a strictly polystable point is a strictly semistable point with a positive dimensional stabilizer group and with a closed orbit in the semistable locus. 

The following lemma shows how $T$-polystability and $G$-polystability are related.

\begin{lemma}\label{lem:torusreduction}
Let $G$ be a reductive group acting linearly on $(X, L)$, and let $T$ be a maximal torus of $G$. Let $x \in X^{ss}$ be a strictly $G$-polystable point. Then there is $g \in G$ such that $gx$ is strictly $T$-polystable.
\end{lemma}

\begin{proof}
Recall that $x \in X$ is strictly $G$-polystable if (1) $x \in X^{ss} \setminus X^s$, (2) $x$ has a positive dimensional stabilizer group, and (3) its orbit $Gx$ is closed in $X^{ss}$. By Theorem \ref{thm:HilbertMumford2}, $x$ is semistable with respect to all maximal tori. The connected component of the stabilizer of $x$ is reductive \cite[Lemma 2.5]{Kir85}. Since it is positive dimensional, it includes a positive dimensional torus $T_1$, and hence there is a maximal torus $T_2 \supset T_1$. Since all maximal tori are conjugate to each other, there is $g \in G$ such that $gT_2 g^{-1} = T$. Then $gx$ is $gT_2 g^{-1}$-semistable and we set $T := gT_2 g^{-1}$. Because $gT_1 g^{-1}$ stabilizes $gx$, $gx$ has a positive dimensional stabilizer. 

Now $Ggx = Gx$ has a closed orbit in $X^{ss}$ if and only if for any one-parameter subgroup $\lambda \subset G$, $\lim_{t \to 0}\lambda(t)x \in Gx$ if the limit exists in $X^{ss}$. In particular, for any one-parameter subgroup $\lambda$ in $T$ such that $\lim_{t \to 0}\lambda(t)gx$ exists, $\lim_{t \to 0}\lambda(t)gx \in Ggx \cap \overline{Tgx} = Tgx$. Thus, $Tgx$ is closed in $X^{ss}$. In summary, $gx$ is strictly $T$-polystable. 
\end{proof}

\begin{remark}\label{rmk:Tpolystabledoesnotimplypolystable}
The converse of Lemma \ref{lem:torusreduction} is not true. Namely, even if $x$ is $T$-polystable for a fixed maximal torus $T \subset G$, it may be possible that $x$ does not have a closed $G$-orbit in $X^{ss}$. See the example in Section \ref{sec:cubic}. 
\end{remark}

Even though $T$-polystability does not provide a complete description of $G$-polystability, it can be regarded as the first step toward the polystability computation.

$T$-polystability has the following combinatorial criterion.

\begin{lemma}\label{lem:Tpolystable}
A point $x \in X^{ss} \setminus X^s$ is $T$-polystable if and only if
\begin{enumerate}
    \item the state $\Xi_x$ is in a positive codimensional linear subspace in $M_\RR$ and;
    \item the trivial character $\chi_0$ is in the relative interior of $\Conv(\Xi_x)$.
\end{enumerate}
\end{lemma}

\begin{proof}
Suppose that $x \in X^{ss}$ is $T$-polystable. Then the identity component of its stabilizer for the $T$-action is a positive dimensional subtorus $U$ of $T$. Thus, every one parameter subgroup in $\Hom(\Bbbk^*, U)$ acts with the same weight, hence $\Xi_x$ must lie on an affine translation of $\Hom(\Bbbk^*, U)^\perp \subsetneq M_\RR$. But since $\Xi_x$ is semistable, it must include $\chi_0$, so it is lying on $A := \Hom(\Bbbk^*, U)^\perp$, proving the first claim.

If $\chi_0$ is on a relative interior of a proper face $Q$ of $\Conv(\Xi_x)$, we may take a supporting hyperplane $\lambda^\perp \subset M_\RR$ such that $Q \subset \lambda^\perp \cap A \ne A$ and $\Xi_x \subset \{ \chi \in M_\RR\;|\; \langle \chi, \lambda \rangle \ge 0\}$. Then with respect to $\lambda \in N_\RR$, $\lim_{t \to 0}\lambda(t)x \in \overline{Tx} \setminus  Tx$
and $x$ is not polystable. If $\chi_0$ is on the outside of $\Conv(\Xi_x)$, then $x$ is $T$-unstable by Theorem \ref{thm:HilbertMumfordII}. Therefore, $\chi_0$ is in the relative interior of $\Conv(\Xi_x)$. The converse is similar.
\end{proof}

The $T$-polystable locus has a stratification. In the following discussion, (semi-)stability is for the $T$-action. 

\begin{definition}\label{def:Tpolystabilitystratification}
Let $A \subset M_\RR$ be a proper linear subspace. Let $Y_A \subset X^{ss}$ be the subset of strictly $T$-polystable points $x$ such that $\Conv(\Xi_x)$ spans $A$ and $\chi_0 \in \intr \Conv(\Xi_x)$. Then for only finitely many $A$, $Y_A$ is nonempty.
Let $\underline{Y}_A$ be the image of $Y_A$ in $X\git T$. Then $\bigsqcup_A \underline{Y}_A$ is a stratification of $(X\git T) \setminus (X^s/T)$. This stratification is called the \emph{$T$-polystable stratification}.
\end{definition}

\begin{remark}\label{rem:flat}
In terms of realizable matroids defined by $\Xi_V$, the stratification is parametrized by the set of non-maximal flats whose convex hull includes the origin in its relative interior.
\end{remark}

Since our eventual interest is the polystable stratification for $X \git_L G$ for a semisimple group $G$-action, we may assume that the state polytope $\Xi_V$ has Weyl group symmetry. Thus, it is sufficient to find the index set
\[
    P_{ps}^F := \{\Xi_A := \Xi_V \cap A\},
\]
where $A$ is a proper subspace of $M_\RR$ such that $\Conv(\Xi_A)$ spans $A$, of the orbits of the Weyl group action. For each index $A$, we may recover a general $T$-polystable point on $Y_A$ by taking $x \in X$ such that $\Xi_x = \Xi_A$.

\begin{question}\label{que:polystable}
Find an algorithm that computes the index set $P_{ps}^F$.%
\end{question}

Let $G$ be a semisimple group and let $x \in X$ be a $T$-polystable point and $\Xi_x = \Xi_A$ for some proper subspace $A \subset M_\RR$. Since a polystable point $x$ is not stable,
its associated state $\Xi_x = \Xi_A$, up to a Weyl group action, must be contained in one of $\Xi_{V, \lambda \ge 0} \in P_s^F$. Moreover, by Lemma \ref{lem:torusreduction}, $\Xi_x \subset \Xi_{V, \lambda = 0}$. 
So we can start from the subset of $\{\Xi_{V, \lambda = 0}\}$ that contains the trivial character $\chi_0$. If $\chi_0 \in \intr \Conv(\Xi_{V, \lambda = 0})$, then $\Xi_{V, \lambda = 0} \in P_{ps}^F$. If $\chi_0$ is on the relative boundary of $\Conv(\Xi_{V, \lambda = 0})$, then by eliminating some characters in $\Xi_{V, \lambda = 0}$, we can find a state that corresponds to a deeper stratum, corresponding to $A' \subsetneq A$. If $\chi_0 \notin \Conv(\Xi_{V, \lambda = 0})$, then $\Xi_{V, \lambda = 0}$ is unstable, so we can discard it and any proper subsets. 

Based on this strategy, we can describe the algorithm for the $T$-polystable stratification.

\begin{algorithm}\label{alg:Polystable}
[Algorithm for the computation of $P_{ps}^{F}$]
\leavevmode\\
\textbf{Input:} The set $P^F_{s}$.\\
\textbf{Output:} The set $P_{ps}^F$.\\
\begin{enumerate}[label=\arabic*.]
\item $\cS_{p} := \emptyset$ %
\item $P_{ps}^F := \emptyset$ %
\item \textbf{for all} $\Xi_{V, \lambda \ge 0} \in P_s^F$ \textbf{do} 
\item \qquad \textbf{if} $\chi_0 \in \intr\Conv(\Xi_{V, \lambda = 0})$ \textbf{then do}
\item \qquad \qquad $P_{ps}^F\ := P_{ps}^F \cup \{\Xi_{V, \lambda = 0}\}$
\item \qquad \textbf{if} $\chi_0 \in \mathrm{Conv}(\Xi_{V, \lambda = 0})$ \textbf{then do}
\item \qquad \qquad $\cS_{p} := \cS_{p} \cup \{\Xi_{V, \lambda = 0}\}$
\item \textbf{for all} $T \in \cS_{p}$ \textbf{do} %
\item \qquad \textbf{for all} $T' \subset T$ \textbf{do}
\item \qquad \qquad \textbf{if} $\dim \Conv(T') < \dim \Conv(T)$ \textbf{do}
\item \qquad \qquad \qquad \textbf{if} $\chi_0 \in \intr \Conv(T')$ \textbf{do}
\item \qquad \qquad \qquad \qquad $P_{ps}^F := P_{ps}^F \cup \{\mathrm{Span} \;T' \cap \Xi_V\}$
\item \textbf{for all} $T \in P_{ps}^F$ \textbf{do} %
\item \qquad \textbf{for all} $g \in W$ \textbf{do}
\item \qquad \qquad \textbf{if} $gT \ne T$ \textbf{and} $gT \in P_{ps}^F$ \textbf{then do}
\item \qquad \qquad \qquad $P_{ps}^F := P_{ps}^F \setminus \{gT\}$
\item \textbf{return} $P_{ps}^F$
\end{enumerate}
\end{algorithm}

We implemented the above algorithm in \texttt{SageMath} \cite{sagemath-code}. 

\begin{remark}\label{rmk:polystablereductive}
The above algorithm can be applied to reductive groups, if we disregard the Weyl group action. More precisely, the input is the set $P_s$ instead of $P^F_s$, and we may skip Lines \texttt{13-16}.
\end{remark}

\begin{remark}
    We implemented these algorithms (for simple groups) in \texttt{SageMath}. The interested reader can find the code with documentation at: \cite{sagemath-code}. In line \texttt{25} in Algorithm \ref{alg:Stable} and in line \texttt{29} in Algorithm \ref{alg:Semistable} (but not in \ref{alg:Polystable}), $W'$ can be replaced by $W$ and the output does not change. However, more unnecessary iterations of the loop will take place and, as the size of $\Xi_V$ grows, this can have a considerable effect in execution time. On the other hand, one needs additional computational time to construct $W'$, so for small problems using $W$ may partially compensate this time. In the current implementation in \cite{sagemath-code} we use $W$ for simplicity of coding.

\end{remark}

\section{Cubic surfaces}\label{sec:cubic}

In this section, we present a classical example (the moduli space of cubic surfaces) to illustrate how the algorithms in Section  \ref{sec:algorithm} work and describe how the outputs can be interpreted in moduli theory. In this case, the GIT stability analysis was first done by Hilbert in \cite{Hil93}. One can find the computation in several modern textbooks, for instance in \cite[Section 7.2.b]{Muk03}. 

Recall that a degree $d$ hypersurface in $\PP^{n+1}$ can be identified with a nonzero section in $\rH^{0}(\PP^{n+1}, \cO_{\PP^{n+1}}(d))$, up to a scalar multiplication. Two hypersurfaces are projectively equivalent if there is a projective automorphism $\mathrm{Aut}(\PP^{n+1}) \cong \PGL_{n+2}$. So the moduli space of $n$-dimensional degree $d$ hypersurfaces is 
\begin{equation}\label{eqn:GITcompactification}
	\PP \rH^{0}(\PP^{n+1}, \cO_{\PP^{n+1}}(d))^* \git \PGL_{n+2} \cong \PP \rH^{0}(\PP^{n+1}, \cO_{\PP^{n+1}}(d))^* \git \SL_{n+2}. 
\end{equation}
The isomorphism is obtained because the scalar matrices in $\SL_{n+2}$ act trivially. 
Because $\SL_{n+2}$ has no torus factor, there is only one linearization of the $\SL_{n+2}$-action and the GIT quotient is uniquely determined \cite[Section 7.2]{Dol03}.

Any smooth $n$-dimensional hypersurface of degree $d > 2$ in $\PP^{n+1}$ is GIT stable \cite[Theorem 10.1]{Dol03}. Thus, the GIT quotient is indeed a compactification of the moduli space of smooth degree $d$ hypersurfaces. 

\begin{definition}
The \emph{GIT compactification} $H_{n,d}$ of the moduli space of $n$-dimensional smooth degree $d$ hypersurfaces is the GIT quotient in \eqref{eqn:GITcompactification}. 
\end{definition}

We now focus on $H_{2,3}$, the moduli space of cubic surfaces. Let $S \cong \Bbbk^{4}$ be the standard $\SL_{4}$-representation. Then $V := \rH^{0}(\PP^{4}, \cO_{\PP^{4}}(3)) \cong \Sym^{3}S \cong \Gamma_{3\omega_{1}}$ is an irreducible $\SL_{4}$-representation whose highest weight is $3\omega_{1}$. For $\SL_{4}$, the rank $d=3$, $N_{\RR} \cong \{(x_{1}, x_{2}, x_{3}, x_{4}) \in \RR^{4}\;|\; \sum x_{i} = 0\}$ and $M_{\RR} \cong \{(y_{1}, y_{2}, y_{3}, y_{4}) \in \RR^{4}\}/(\sum y_{i})$. The pairing $\langle \;, \; \rangle : N_{\RR} \times M_{\RR} \to \RR$ is induced from the standard dot product of $\RR^4$. By using the pairing, we may identify $N_{\RR}$ and $M_\RR$. For $V \cong \Gamma_{3\omega_{1}}$, 
\[
	\Xi_{V} = \{(y_{1}, y_{2}, y_{3}, y_{4}) \in M_{\RR}\;|\; y_{i} \in \ZZ_{\ge 0}, \sum y_{i} = 3\},
\]
which are 20 lattice points in a regular tetrahedron (Figure \ref{fig:stateforcubic}) in $M_\RR$.
The fundamental chamber (technically, in $N_\RR$) is drawn as a grey simplicial cone.

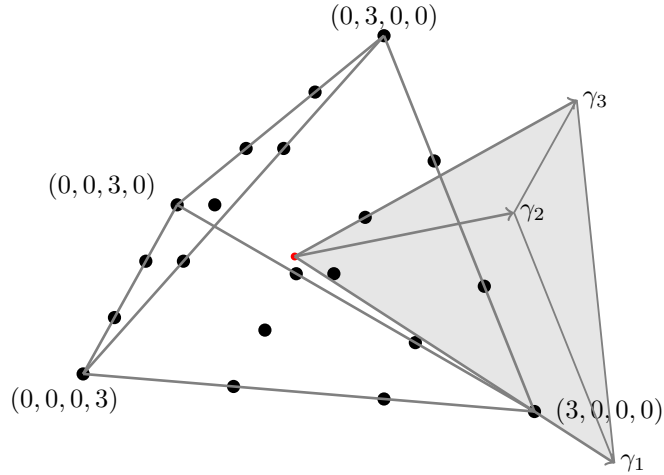
\begin{figure}[!ht]
\def\Px{12}
\def\Py{-1}
\def\Qx{8}
\def\Qy{9}
\def\Rx{2.5}
\def\Ry{4.5}
\begin{tikzpicture}[scale=0.5]
	\fill[fill=gray!20] (\Px/4+\Qx/4+\Rx/4, \Py/4+\Qy/4+\Ry/4) -- (\Px-\Qx/3-\Rx/3+\Px/4+\Qx/4+\Rx/4, \Py-\Qy/3-\Ry/3+\Py/4+\Qy/4+\Ry/4) -- (\Px/3+\Qx/3+\Rx/3+\Px/4+\Qx/4+\Rx/4, \Py/3+\Qy/3+\Ry/3+\Py/4+\Qy/4+\Ry/4);
	\fill (0, 0) circle (5pt);
	\node at (-0.5, -0.7) {$(0,0,0,3)$};
	\fill (\Px/3, \Py/3) circle (5pt);
	\fill (\Px*2/3, \Py*2/3) circle (5pt);
	\fill (\Px, \Py) circle (5pt);
	\node at (\Px+2, \Py) {$(3,0,0,0)$};
	\fill (\Qx/3, \Qy/3) circle (5pt);
	\fill (\Qx*2/3, \Qy*2/3) circle (5 pt);
	\fill (\Qx, \Qy) circle (5pt);
	\node at (\Qx, \Qy+0.5) {$(0,3,0,0)$};
	\fill (\Rx/3, \Ry/3) circle (5pt);
	\fill (\Rx*2/3, \Ry*2/3) circle (5pt);
	\fill (\Rx, \Ry) circle (5pt);
	\node at (\Rx-2,\Ry+0.5) {$(0,0,3,0)$};
	\fill (\Px/3+\Qx/3, \Py/3+\Qy/3) circle (5pt);
	\fill (\Px*2/3+\Qx/3, \Py*2/3+\Qy/3) circle (5pt);
	\fill (\Px/3+\Qx*2/3, \Py/3+\Qy*2/3) circle (5pt);
	\fill (\Px/3+\Rx/3, \Py/3+\Ry/3) circle (5pt);
	\fill (\Px*2/3+\Rx/3, \Py*2/3+\Ry/3) circle (5pt);
	\fill (\Px/3+\Rx*2/3, \Py/3+\Ry*2/3) circle (5pt);
	\fill (\Qx/3+\Rx/3, \Qy/3+\Ry/3) circle (5pt);
	\fill (\Qx*2/3+\Rx/3, \Qy*2/3+\Ry/3) circle (5pt);
	\fill (\Qx/3+\Rx*2/3, \Qy/3+\Ry*2/3) circle (5pt);
	\fill (\Px/3+\Qx/3+\Rx/3, \Py/3+\Qy/3+\Ry/3) circle (5pt);
	\fill[red] (\Px/4+\Qx/4+\Rx/4, \Py/4+\Qy/4+\Ry/4) circle (3pt);
	\draw[line width=1pt, gray] (0,0) -- (\Px,\Py);
	\draw[line width=1pt, gray] (0,0) -- (\Qx,\Qy);
	\draw[line width=1pt, gray] (0,0) -- (\Rx,\Ry);
	\draw[line width=1pt, gray] (\Qx,\Qy) -- (\Px,\Py);
	\draw[line width=1pt, gray] (\Rx,\Ry) -- (\Qx,\Qy);
	\draw[line width=1pt, gray] (\Px,\Py) -- (\Rx,\Ry);
	\draw[->][line width=1pt,gray] (\Px/4+\Qx/4+\Rx/4, \Py/4+\Qy/4+\Ry/4) -- (\Px-\Qx/3-\Rx/3+\Px/4+\Qx/4+\Rx/4, \Py-\Qy/3-\Ry/3+\Py/4+\Qy/4+\Ry/4);
	\node at (\Px-\Qx/3-\Rx/3+\Px/4+\Qx/4+\Rx/4+0.5, \Py-\Qy/3-\Ry/3+\Py/4+\Qy/4+\Ry/4) {$\gamma_{1}$};
	\draw[->][line width=1pt,gray] (\Px/4+\Qx/4+\Rx/4, \Py/4+\Qy/4+\Ry/4) -- (\Px/3+\Qx/3-\Rx/3+\Px/4+\Qx/4+\Rx/4, \Py/3+\Qy/3-\Ry/3+\Py/4+\Qy/4+\Ry/4);
	\node at (\Px/3+\Qx/3-\Rx/3+\Px/4+\Qx/4+\Rx/4+0.5, \Py/3+\Qy/3-\Ry/3+\Py/4+\Qy/4+\Ry/4) {$\gamma_{2}$};
	\draw[->][line width=1pt,gray] (\Px/4+\Qx/4+\Rx/4, \Py/4+\Qy/4+\Ry/4) -- (\Px/3+\Qx/3+\Rx/3+\Px/4+\Qx/4+\Rx/4, \Py/3+\Qy/3+\Ry/3+\Py/4+\Qy/4+\Ry/4);
	\node at (\Px/3+\Qx/3+\Rx/3+\Px/4+\Qx/4+\Rx/4+0.5, \Py/3+\Qy/3+\Ry/3+\Py/4+\Qy/4+\Ry/4) {$\gamma_{3}$};
	\draw[line width=0.7pt,gray] (\Px-\Qx/3-\Rx/3+\Px/4+\Qx/4+\Rx/4, \Py-\Qy/3-\Ry/3+\Py/4+\Qy/4+\Ry/4) -- (\Px/3+\Qx/3-\Rx/3+\Px/4+\Qx/4+\Rx/4, \Py/3+\Qy/3-\Ry/3+\Py/4+\Qy/4+\Ry/4);
	\draw[line width=0.7pt,gray] (\Px-\Qx/3-\Rx/3+\Px/4+\Qx/4+\Rx/4, \Py-\Qy/3-\Ry/3+\Py/4+\Qy/4+\Ry/4) -- (\Px/3+\Qx/3+\Rx/3+\Px/4+\Qx/4+\Rx/4, \Py/3+\Qy/3+\Ry/3+\Py/4+\Qy/4+\Ry/4);
	\draw[line width=0.7pt,gray] (\Px/3+\Qx/3-\Rx/3+\Px/4+\Qx/4+\Rx/4, \Py/3+\Qy/3-\Ry/3+\Py/4+\Qy/4+\Ry/4) -- (\Px/3+\Qx/3+\Rx/3+\Px/4+\Qx/4+\Rx/4, \Py/3+\Qy/3+\Ry/3+\Py/4+\Qy/4+\Ry/4);
\end{tikzpicture}
\caption{State $\Xi_{V}$ and the fundamental chamber $F$}\label{fig:stateforcubic}
\end{figure}

The Weyl group $W$ is isomorphic to $S_{4}$ and its action on all latices/vector spaces is induced by its natural permutation action on the four coordinates of $N_{\RR}$ and $M_{\RR}$. The fundamental chamber in $N_{\RR}$ is 
\[
	F = \{(x_{1}, x_{2}, x_{3}, x_{4}) \in \RR^{4}\;|\; x_{1} \ge x_{2} \ge x_{3} \ge x_{4}, \sum x_{i} = 0\}
\]
and it is a three-dimensional simplicial cone generated by $\gamma_{1} = (3, -1, -1, -1)$, $\gamma_{2} = (1, 1, -1, -1)$, and $\gamma_{3} = (1, 1, 1, -3)$. Figure \ref{fig:essentialpart} shows the set $A_1 := \bigcup_{i=1}^{3}\Xi_{V, \gamma_{i} \ge 0} \setminus \bigcap_{i=1}^{3}\Xi_{V, \gamma_{i} > 0}$ in \eqref{eqn:differenceforstability}. Ten white circles
are excluded and $|A_1| = 10$. Finally, there are two pairs of vertices (each pair contains one of two remaining extremal vertices) which are proportional. Thus, the set $\Xi_{V}^{E,s}$ of essential characters is 
\[
    \Xi_V^{E,s} = \{(2, 0, 0, 1), (1, 1, 0, 1), (1, 0, 2, 0), (1, 0, 1, 1), (1, 0, 0, 2), (0, 2, 1, 0), (0, 2, 0, 1), (0, 1, 2, 0)\}.
\]

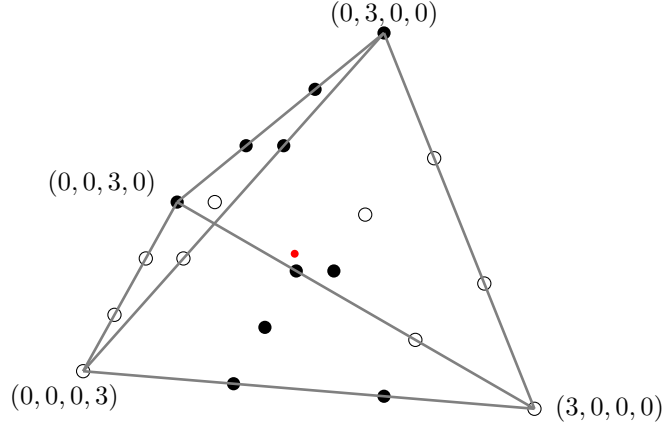
\begin{figure}[!ht]
\def\Px{12}
\def\Py{-1}
\def\Qx{8}
\def\Qy{9}
\def\Rx{2.5}
\def\Ry{4.5}
\begin{tikzpicture}[scale=0.5]
    \draw (0, 0) circle (5pt);
	\node at (-0.5, -0.7) {$(0,0,0,3)$};
	\fill (\Px/3, \Py/3) circle (5pt);
	\fill (\Px*2/3, \Py*2/3) circle (5pt);
	\draw (\Px, \Py) circle (5pt);
	\node at (\Px+2, \Py) {$(3,0,0,0)$};
	\draw (\Qx/3, \Qy/3) circle (5pt);
	\fill (\Qx*2/3, \Qy*2/3) circle (5 pt);
	\fill (\Qx, \Qy) circle (5pt);
	\node at (\Qx, \Qy+0.5) {$(0,3,0,0)$};
	\draw (\Rx/3, \Ry/3) circle (5pt);
	\draw (\Rx*2/3, \Ry*2/3) circle (5pt);
	\fill (\Rx, \Ry) circle (5pt);
	\node at (\Rx-2,\Ry+0.5) {$(0,0,3,0)$};
	\fill (\Px/3+\Qx/3, \Py/3+\Qy/3) circle (5pt);
	\draw (\Px*2/3+\Qx/3, \Py*2/3+\Qy/3) circle (5pt);
	\draw (\Px/3+\Qx*2/3, \Py/3+\Qy*2/3) circle (5pt);
	\fill (\Px/3+\Rx/3, \Py/3+\Ry/3) circle (5pt);
	\draw (\Px*2/3+\Rx/3, \Py*2/3+\Ry/3) circle (5pt);
	\fill (\Px/3+\Rx*2/3, \Py/3+\Ry*2/3) circle (5pt);
	\draw (\Qx/3+\Rx/3, \Qy/3+\Ry/3) circle (5pt);
	\fill (\Qx*2/3+\Rx/3, \Qy*2/3+\Ry/3) circle (5pt);
	\fill (\Qx/3+\Rx*2/3, \Qy/3+\Ry*2/3) circle (5pt);
	\draw (\Px/3+\Qx/3+\Rx/3, \Py/3+\Qy/3+\Ry/3) circle (5pt);
	\fill[red] (\Px/4+\Qx/4+\Rx/4, \Py/4+\Qy/4+\Ry/4) circle (3pt);
	\draw[line width=1pt, gray] (0,0) -- (\Px,\Py);
	\draw[line width=1pt, gray] (0,0) -- (\Qx,\Qy);
	\draw[line width=1pt, gray] (0,0) -- (\Rx,\Ry);
	\draw[line width=1pt, gray] (\Qx,\Qy) -- (\Px,\Py);
	\draw[line width=1pt, gray] (\Rx,\Ry) -- (\Qx,\Qy);
	\draw[line width=1pt, gray] (\Px,\Py) -- (\Rx,\Ry);
\end{tikzpicture}
\caption{The set $\bigcup_{i=1}^{3}\Xi_{V, \gamma_{i} \ge 0} \setminus \bigcap_{i=1}^{3}\Xi_{V, \gamma_{i} > 0}$}
\label{fig:essentialpart}
\end{figure}

For each pair $\{\chi_{1}, \chi_{2}\} \subset \Xi_{V}^{E}$, we compute a (unique up to scalar multiple) one-parameter subgroup $\lambda$ with $\langle \lambda, \chi_{i}\rangle = 0$. If $\lambda \in F$, record $\Xi_{V, \lambda \ge 0}$ and compute maximal elements among them. There are three maximal sets, corresponding to three one-parameter subgroups $\lambda_{1} = (1, 0, 0, -1)$, $\lambda_{2} = (2, 0, -1, -1)$, and $\lambda_{3} = (1, 1, 0, -2)$:
\[
\begin{split}
	\Xi_{V, \lambda_{1} \ge 0} =& \{ (2, 1, 0, 0) , (1, 1, 0, 1) , (2, 0, 0, 1) , 
(0, 1, 2, 0) ,  (1, 0, 2, 0) , (0, 2, 1, 0) , (0, 3, 0, 0) ,\\
& (3, 0, 0, 0) , (1, 0, 1, 1) , (2, 0, 1, 0) ,  (1, 1, 1, 0) , (1, 2, 0, 0) , (0, 0, 3, 0) \},\\
	\Xi_{V, \lambda_{2} \ge 0} =& \{ (2, 1, 0, 0) , (1, 1, 0, 1) , (2, 0, 0, 1) , 
(1, 1, 1, 0) , (0, 3, 0, 0) , (3, 0, 0, 0) , (1, 0, 1, 1) ,\\
& (2, 0, 1, 0) , 
(1, 2, 0, 0) , (1, 0, 2, 0) , (1, 0, 0, 2) \},\\
	\Xi_{V, \lambda_{3} \ge 0} =& \{(2, 1, 0, 0) , (0, 2, 0, 1) , (1, 1, 0, 1) , 
(2, 0, 0, 1) , (0, 1, 2, 0) , (1, 0, 2, 0) , (0, 2, 1, 0) ,\\
& (0, 3, 0, 0) , 
(3, 0, 0, 0) , (2, 0, 1, 0) ,  (1, 1, 1, 0) , (1, 2, 0, 0) , (0, 0, 3, 0)  \}.
\end{split}
\]

Now we turn to a geometric interpretation. In this example, each character $\chi = (y_{0}, y_{1}, y_{2}, y_{3})$ can be identified with a monomial $\prod X_{i}^{y_{i}}$, where $(X_{0}, X_{1}, X_{2}, X_{3})$ is a fixed homogeneous coordinate of $\PP^{3}$. Then for $\lambda_{1} = (1, 0, 0, -1)$, a general polynomial associated to the maximal state $\Xi_{V, \lambda_{1} \ge 0}$ is of the form 
\[
	X_{0}X_{3}f_{1}(X_{0}, X_{1}, X_{2}) + f_{3}(X_{0}, X_{1}, X_{2}),
\]
where $f_{d}$ is a degree $d$ homogeneous polynomial. At the point $P := (0, 0, 0, 1) \in \PP^{3}$, the zero locus has the tangent cone $X_0 f_1 (X_0, X_1, X_2)$, that is a union of two planes. Similarly, for $\Xi_{V, \lambda_{2} \ge 0}$, we have 
\[
	cX_{1}^{3} + X_{0}f_{2}(X_{0}, X_{1}, X_{2}, X_{3}), 
\]
with $c \in \Bbbk$. Then the surface contains a line $X_0 = X_1 = 0$, and a singular point on it, where the tangent cone contains a plane $X_0 = 0$. Finally, for $\Xi_{V, \lambda_{3} \ge 0}$, we obtain
\[
	X_{3}f_{2}(X_{0}, X_{1}) + f_{3}(X_{0}, X_{1}, X_{2}).
\]
At $P$, the surface has the tangent cone $f_2 (X_0, X_1)$, that is a quadric of rank two. Recall that an ordinary double point of a surface is a singular point where the tangent cone is a full rank quadratic cone. The above computation tells us that if a cubic surface is not stable, then it has a singular point which is not an ordinary double point. The reader can verify that our outcome recovers the equations in \cite[p.227]{Muk03}. Indeed, the converse is also true \cite[Theorem 7.14]{Muk03}.

For the semistable locus, we need to compute $\bigcup_{i=1}^{3}\Xi_{V, \gamma_{i} > 0} \setminus K_{nm}$ in \eqref{eqn:essentialsetforss}, where $K_{nm}$ is the set of non-minimal elements in $\bigcap_{i=1}^{3}\Xi_{V, \gamma_{i} > 0}$. It is straightforward to see that $\bigcup_{i=1}^{3}\Xi_{V, \gamma_{i} > 0} = \bigcup_{i=1}^3 \Xi_{V, \gamma_i \ge 0}$. By a direct calculation of the paring $\langle \gamma_{i}, \chi_{j}\rangle$, we can see that $\left|\bigcap_{i=1}^{3}\Xi_{V, \gamma_{i} > 0}\right| = 5$ and there is a unique minimal element $(1, 1, 1, 0)$. Thus, for the semistable locus calculation, 
\[
\begin{split}
    \Xi_V^{E, ss} &= \{(2, 0, 0, 1), (1, 1, 1, 0), (1, 1, 0, 1), (1, 0, 2, 0), (1, 0, 1, 1), (1, 0, 0, 2), (0, 3, 0, 0),\\
    & \quad (0, 2, 1, 0), (0, 2, 0, 1), (0, 1, 2, 0), (0, 0, 3, 0)\}.
\end{split}
\]

By using all triples of characters in $\Xi_V^{E, ss}$, with Algorithm \ref{alg:Semistable}, we obtain three one-parameter subgroups $\mu_{1} = (3, -1, -1, -1)$, $\mu_{2} = (5, 1, 1, -7)$, $\mu_{3} = (3, 3, -1, -5)$ which correspond to maximal unstable states. These are:
\[
\begin{split}
	\Xi_{V, \mu_{1} > 0} =& \{ (2, 1, 0, 0) , (2, 0, 0, 1) , (1, 1, 0, 1) , 
(3, 0, 0, 0) , (1, 2, 0, 0) , (1, 0, 1, 1) ,\\
& (2, 0, 1, 0) , (1, 0, 2, 0) ,
 (1, 1, 1, 0) , (1, 0, 0, 2)\},\\
	\Xi_{V, \mu_{2} > 0} =& \{ (2, 1, 0, 0) , (2, 0, 0, 1) , (1, 0, 2, 0) , 
(1, 1, 1, 0) , (0, 1, 2, 0) , (3, 0, 0, 0) , \\
&(0, 3, 0, 0) , (2, 0, 1, 0) , 
(0, 2, 1, 0) , (1, 2, 0, 0) , (0, 0, 3, 0) \},\\
	\Xi_{V, \mu_{3} > 0} =& \{(2, 1, 0, 0) , (2, 0, 1, 0) , (1, 1, 0, 1) , 
(2, 0, 0, 1) , (1, 0, 2, 0) , (1, 1, 1, 0) , \\
&(0, 1, 2, 0) , (3, 0, 0, 0) ,
 (0, 3, 0, 0) , (0, 2, 0, 1) , (0, 2, 1, 0) , (1, 2, 0, 0)  \}.
\end{split}
\]

In \cite[Prop 7.22]{Muk03}, for the semistability computation, he found the one parameter subgroups
\[
    \nu_1 := (3, -1, -1, -1), \quad \nu_2 := (3, 3, -1, -5), \quad \nu_3 := (3, 1, 1, -5). 
\]
Thus, the list of the one parameter subgroups are different. However, one can check that $\Xi_{\mu_2 > 0} = \Xi_{\nu_3 > 0}$.
By analyzing the equations of unstable cubic surfaces associated to $\Xi_{V, \mu_i > 0}$, one can conclude that a semistable cubic surface may have one extra class of singularities than those appearing for stable surfaces --- a double point whose tangent cone is the union of two planes and the intersection of the planes does not lie on the surface. For the details, consult \cite[Theorem 7.20]{Muk03}.

Finally, by using Algorithm \ref{alg:Polystable}, we can describe the $T$-polystable stratification. There are five strata in total. For each $\lambda_i$, the associated $T$-polystable state is 
\[
    \Xi_{V, \lambda_1 = 0} = \{(1,1,0,1),(1,0,1,1),(0,0,3,0),(0,3,0,0),(0,1,2,0),(0,2,1,0)\},
\]
\[
    \Xi_{V,\lambda_2 = 0} = \{(0,3,0,0),(1,0,1,1),(1,0,2,0),(1,0,0,2)\},
\]
\[
    \Xi_{V, \lambda_3 = 0} = \{(0,2,1,0),(1,1,0,1),(2,0,0,1),(0,0,3,0)\}.
\]
The dimension of each convex hull is two. In $\Xi_{V, \lambda_1 = 0}$, there are two subsets, whose convex hulls are one-dimensional, and contain the origin (which, due to the description of $M_{\mathbb R}$ as a quotient, corresponds to
a scalar multiple of $(1,1,1,1)$): $\Xi_2 := \{(1,1,0,1), (0,0,3,0)\}$ and $\Xi_3 := \{(1,0,1,1), (0,3,0,0)\}$. $\Xi_i$ is also contained in $\Xi_{V, \lambda_i = 0}$ and is the only one in it. On the moduli space $H_{2, 3}$, two strata associated to $\Xi_2$ and $\Xi_3$ are identified by the $\SL_4$-action and their image in $H_{2,3}$ is a point. Indeed, $\Xi_2$ and $\Xi_3$ are in the same orbit of the $W$-action. On the other hand, the larger dimensional strata are not closed with respect to the $\SL_4$-action, hence they are not $\SL_4$-polystable (Remark \ref{rmk:Tpolystabledoesnotimplypolystable}). Therefore, the GIT quotient $H_{2,3}$ is a one-point compactification of the quotient of the stable locus \cite[Theorem 7.20]{Muk03}.

\section{Examples and statistics}\label{sec:examplesandstatistics}

\subsection{Statistics}\label{subsec:stat}

In Table \ref{table:statistics} we present statistics obtained from running Algorithms \ref{alg:Stable}, \ref{alg:Semistable}, and \ref{alg:Polystable}. Most of these statistics were obtained using our \texttt{SageMath} implementation, except for the genus 7 Mukai model, which was computed using \texttt{C++} instead. In the table, we cite a reference for the results that we have found in the literature. (Our citations are not necessarily to the first appearance, especially for the classical GIT problems.) In the subsections following the table, we comment on some of the results that are, to our knowledge, new.

For each example in the table, we give the following data.
\begin{itemize}
\item a short description of the GIT problem;
\item the root system and representation. Here $V(\lambda)$ denotes the irreducible representation with highest weight $\lambda$, and $\omega_i$ are the fundamental dominant weights for this root system; 
\item the run times for Algorithms \ref{alg:Stable}, \ref{alg:Semistable}, and \ref{alg:Polystable} in seconds (unless otherwise indicated); 
\item the size of the set $\Xi_V$, which serves as a measure of the complexity of the input;  
\item the size of the set $A_3$ computed in Algorithm \ref{alg:Stable};
\item the size of the set $B_2$ computed in Algorithm \ref{alg:Semistable};
\item the sizes of the output sets $P_s^F$, $P_{ss}^F$, and  $P_{ps}^F$
\end{itemize}

Our current code implements Algorithms \ref{alg:Stable}, \ref{alg:Semistable}, and \ref{alg:Polystable} faithfully except in one aspect: in Algorithm \ref{alg:Stable} line 25 and Algorithm \ref{alg:Semistable} line 29, the code uses the full Weyl group $W$ instead of the subset $W'$. This should give identical output. This saved us programming time at the cost of additional computing time. 

Unless otherwise indicated, the examples were run on 
an Amazon Web Services \texttt{c4.2xlarge} instance to allow for comparison of the running times. Each such instance had 8 vCPUs, 2.9GHz processors and 15GB memory, and 24GB storage.

In Figure \ref{fig:statistics} we plotted selected data from Table \ref{table:statistics}. Specifically, we plotted the run time for Algorithm \ref{alg:Stable} and the size of its output $|P_s^F|$ for four series of examples: hypersurfaces in $\mathbb{P}^2$, $\mathbb{P}^3$, and $\mathbb{P}^4$, and cubic hypersurfaces of dimensions 1--5. 

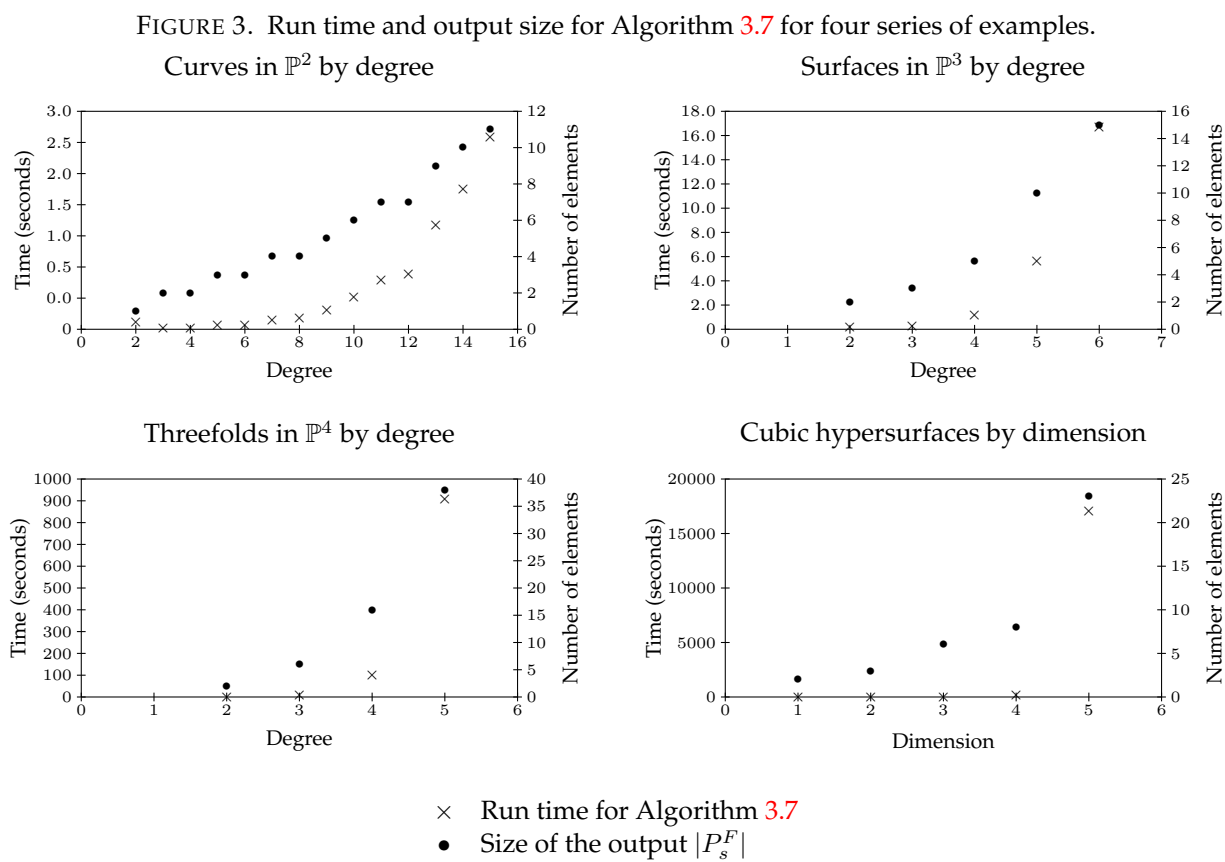
\small
\begin{center}
\begin{table}[!ht]
\caption{Statistics for Algorithms \ref{alg:Stable}, \ref{alg:Semistable}, and \ref{alg:Polystable} via our  \texttt{SageMath} implementation.}  \label{table:statistics} 
\begin{tabular}{l|ll|rrr|rrr|rrr}
Description & Type & Rep. & & Time & (sec) & $|\Xi_V|$ & $|A_3|$ & $|B_2|$ & $|P_s^F|$ & $|P_{ss}^F|$ & $|P_{ps}^F|$\\ \hline
Algorithm & &  & {\ref{alg:Stable}} & {\ref{alg:Semistable}} &{\ref{alg:Polystable}} &  &  &  &  &  & \\ \hline

  Plane curves &&&&&&&&\\
degree 2& A2 & $\Gamma_{2 \omega_1}$ & 0.114 &  0.024 &  0.019 &  6 &  1 &  4 &  1 &  2 &  1\\
{\color{white}degree} 3 \cite{HM98} & A2 & $\Gamma_{3 \omega_1}$ & 0.025 &  0.018 &  0.028 &  10 &  3 &  5 &  2 &  1 &  2\\
{\color{white}degree} 4 \cite{AA23} & A2 & $\Gamma_{4 \omega_1}$ & 0.027 &  0.06 &  0.009 &  15 &  3 &  8 &  2 &  2 &  1\\
{\color{white}degree} 5 \cite[p. 80]{MFK94} & A2 & $\Gamma_{5 \omega_1}$ & 0.065 &  0.136 &  0.01 &  21 &  5 &  11 &  3 &  3 &  1\\
{\color{white}degree} 6 \cite{Shah-sextics}& A2 & $\Gamma_{6\omega_1}$  & 0.076 &  0.145 &  0.058 &  28 &  5 &  13 &  3 &  2 &  3\\
{\color{white}degree} 7 & A2 & $\Gamma_{7 \omega_1}$ & 0.155 &  0.352 &  0.016 &  36 &  9 &  17 &  4 &  4 &  1\\
{\color{white}degree} 8 & A2 & $\Gamma_{8 \omega_1}$ & 0.18 &  0.559 &  0.031 &  45 &  9 &  21 &  4 &  4 &  2\\
{\color{white}degree} 9& A2 & $\Gamma_{9 \omega_1}$ & 0.311 &  0.661 &  0.177 &  55 &  9 &  24 &  5 &  4 &  4\\
{\color{white}degree} 10 & A2 & $\Gamma_{10 \omega_1}$ & 0.52 &  1.242 &  0.059 &  66 &  13 &  29 &  6 &  6 &  3\\
{\color{white}degree} 11 & A2 & $\Gamma_{11 \omega_1}$ & 0.798 &  1.836 &  0.052 &  78 &  19 &  34 &  7 &  7 &  2\\
{\color{white}degree} 12 & A2 & $\Gamma_{12 \omega_1}$ & 0.882 &  2.116 &  0.616 &  91 &  13 &  38 &  7 &  6 &  5\\
{\color{white}degree} 13 & A2 & $\Gamma_{13 \omega_1}$ & 1.681 &  3.658 &  0.094 &  105 &  25 &  44 &  9 &  9 &  3\\
{\color{white}degree} 14 & A2 & $\Gamma_{14 \omega_1}$ & 2.25 &  5.032 &  0.17 &  120 &  25 &  50 &  10 &  10 &  5\\
{\color{white}degree} 15 & A2 & $\Gamma_{15 \omega_1}$ & 3.087 &  5.974 &  2.205 &  136 &  21 &  55 &  11 &  10 &  7\\ \hline
  Surfaces &&&&&&&&\\
  Quadrics & A3 & $\Gamma_{2 \omega_1}$ &  0.191 &  0.128 &  0.078 &  10 &  4 &  7 &  2 &  2 &  3\\
  Cubics \cite[\S 7.2.b]{Muk03}& A3 & $\Gamma_{3 \omega_1}$ &  0.277 &  0.886 &  0.122 &  20 &  8 &  15 &  3 &  3 &  3\\
 Quartics \cite{Shah81}& A3 & $\Gamma_{4 \omega_1}$ & 1.181 &  2.377 &  2.196 &  35 &  17 &  21 &  5 &  3 &  7\\
  Quintics \cite{Gal19}& A3 & $\Gamma_{5 \omega_1}$ & 5.647 &  19.259 &  2.808 &  56 &  26 &  37 &  10 &  11 &  4\\
  Sextics & A3 & $\Gamma_{5 \omega_1}$ & 16.736 &  69.257 &  99.363 &  84 &  29 &  54 &  15 &  18 &  13\\ \hline
  Threefolds&&&&&&&&\\ 
Quadrics & A4 & $\Gamma_{2 \omega_1}$ & 0.795 &  1.813 &  0.438 &  15 &  9 &  12 &  2 &  3 &  3\\
Cubics \cites{All03,Yok02}& A4 & $\Gamma_{3 \omega_1}$ & 8.999 &  38.812 &  7.657 &  35 &  21 &  28 &  6 &  6 &  4\\
Quartics & A4 & $\Gamma_{4 \omega_1}$ & 101.999 &  814.157 &  4222.452 &  70 &  39 &  56 &  16 &  23 &  15\\
Quintics \cite{lakhani2010git}& A4 & $\Gamma_{5 \omega_1}$ & 907.658 & 5769.867 & $\text{{---}}^{a}$ &	126	& 76& 84 &	38	&57	& ---\\ \hline
More Cubics &&&&&&&&&\\
4folds \cites{Laz09,Yok08} & A5 & $\Gamma_{3\omega_1}$ &  178.57 &  2020.521 &  24235.369 &  56 &  34 &  44 &  8 &  10 &  14\\
5folds \cite{Shi14}& A6 & $\Gamma_{3\omega_1}$ & 17052.308 & $\text{{---}}^{b}$ & --- & 84 & 60  & 72 & 23 & ---& ---\\ \hline
  Pencils of quadrics &&&&&&&&\\
in $\PP^2$ \cite{Papazachariou-thesis} & A2 & $\Wedge^2 \Gamma_{2\omega_1}$ & 0.025 &  0.048 &  0.007 &  12 &  3 &  7 &  2 &  2 &  1\\
{\color{white} in} $\PP^3$ \cite{Papazachariou-thesis} & A3 & $\Wedge^2 \Gamma_{2\omega_1}$  &1.076 &  1.524 &  1.955 &  31 &  15 &  18 &  5 &  3 &  7\\
{\color{white} in} $\PP^4$ \cite{AM99}& A4 & $\Wedge^2 \Gamma_{2\omega_1}$ & 92.536 &  598.388 &  570.602 &  65 &  39 &  52 &  16 &  22 &  12\\
{\color{white} in} $\PP^5$ & A5 & $\Wedge^2 \Gamma_{2\omega_1}$ & 12424.738 & $\text{{---}}^{b}$& ---& 120 & 73 & 98 & 57 &  ---& ---\\ \hline
  Nets of quadrics &&&&&&&&\\
in $\PP^2$ & A2 & $\Wedge^3 \Gamma_{2 \omega_1}$ &  0.162 & 0.05 & 0.126 & 13 & 3 & 7 & 3 & 2 & 3 \\
{\color{white} in} $\PP^3$ & A3 & $\Wedge^3 \Gamma_{2 \omega_1}$ &  6.981 & 24.104 & 52.355 & 56 & 17 & 37 & 11 & 14 & 10 \\
{\color{white} in} $\PP^4$ \cite{FS13} & A4 & $\Wedge^3 \Gamma_{2 \omega_1}$ & 25993.535 & $57726.588^{c}$ & ---  & 165 & 93 & 124 & 196  & 268 & ---   \\ \hline
  Pencils of cubics &&&&&&&&\\
in $\PP^2$ \cite{Mir80} & A2 & $\Wedge^2 \Gamma_{3 \omega_1}$ & 0.181 & 0.124 & 0.126 & 25 & 5 & 12 & 3 &2& 3\\ \hline
$\Gamma_{\omega_3}$ & A4 & $\Gamma_{\omega_3}$ & 0.482 &  0.539 &  0.108 &  10 &  6 &  8 &  2 &  2 &  1\\
& B4 & $\Gamma_{\omega_3}$ & 60.775 &  131.591 & $\text{{---}}^{a}$ &  65 &  22 &  36 &  8 &  7 & --- \\
& C4 & $\Gamma_{\omega_3}$ & 23.353 &  46.532 &  40.812 &  40 &  13 &  24 &  6 &  7 &  13\\
& D4 & $\Gamma_{\omega_3}$ &  0.295 &  0.524 &  0.481 &  8 &  3 &  5 &  1 &  2 &  3\\
  \hline
  
  Byun-Lee &&&&&&&&\\
$d=3$ \cite{BL15}& B2 & $\Gamma_{3\omega_1}$ &  0.159 &  0.096 &  0.097 &  25 &  3 &  10 &  3 &  2 &  4\\
$d=4$ & B2 & $\Gamma_{4\omega_1}$ & 0.195 &  0.262 &  0.298 &  41 &  4 &  15 &  4 &  3 &  5\\
$d=5$ & B2 & $\Gamma_{5\omega_1}$ & 0.584 &  0.771 &  1.046 &  61 &  6 &  22 &  6 &  5 &  7\\
$d=6$ & B2 & $\Gamma_{6\omega_1}$ & 1.055 &  1.575 &  4.113 &  85 &  7 &  29 &  7 &  6 &  8\\
$d=7$ & B2 & $\Gamma_{7\omega_1}$ & 2.716 &  3.553 &  16.417 &  113 &  10 &  38 &  10 &  9 &  11\\
$d=8$ & B2 & $\Gamma_{8\omega_1}$ & 5.017 &  6.486 &  67.564 &  145 &  12 &  47 &  12 &  11 &  13\\ \hline
  Mukai Problems &&&&&&&&\\
Genus 7 & D5 & $\Wedge^7\Gamma_{\omega_4}$ & --- & --- & --- & 1456 & 852 & 1026  & --- & --- & ---\\
Genus 8 & A5 & $\Wedge^8\Gamma_{\omega_2}$ & --- & --- & --- & 1086 &739 & 863 & ---& ---& ---\\
  Genus 9 & C3 & $\Wedge^9\Gamma_{\omega_3}$ & 7079.337 & 13324.478 & $\text{{---}}^{a}$ &  242 &  51 &  120 &  142 & 186 & --- \\ \hline
\end{tabular}\\
\mbox{}\\
{---} \quad Not attempted; \qquad $\text{{---}}^{a}$ \quad Stopped after 48 hours; \qquad $\text{{---}}^{b}$ \quad Out of memory; \qquad ${}^{c}$ \quad Ran on AWS r5 instance 
\end{table}
\end{center}
\normalsize

\begin{center}
\begin{figure}[!ht]
\caption{Run time and output size for Algorithm \ref{alg:Stable} for four series of examples.}  
\label{fig:statistics}
  \begin{tikzpicture}[scale=5.8]
    \begin{scope}[xshift=0,yshift=24]
    \draw (0.5,0.6) node {Curves in $\mathbb{P}^2$ by degree};
    \draw (0,0)--(1,0)--(1,0.5)--(0,0.5)--(0,0);
    \draw (0.5,-0.1) node {\footnotesize Degree};
    \draw (-0.13,0.25) node[rotate=90] {\footnotesize Time (seconds)};
    \draw (1.12,0.25) node[rotate=90] {\footnotesize Number of elements};
\foreach \x in {0,0.125, 0.25, 0.375, 0.5, 0.625, 0.75, 0.875,1}
  \draw (\x,-0.01) -- (\x,0.01);
\draw (0,-0.03) node {\tiny $0$};
\draw (0.125, -0.03) node {\tiny $2$};
\draw (0.25, -0.03) node {\tiny $4$};
\draw (0.375, -0.03) node {\tiny $6$};
\draw (0.5, -0.03) node {\tiny $8$};
\draw (0.625, -0.03) node {\tiny $10$};
\draw (0.75, -0.03) node {\tiny $12$};
\draw (0.875, -0.03) node {\tiny $14$};
\draw (1,-0.03) node {\tiny $16$};
\foreach \y in {0, 0.071428571, 0.142857143, 0.214285714, 0.285714286, 0.357142857, 0.428571429, 0.5}
\draw (-0.01,\y) -- (0.01,\y);
\draw (0,0) node[anchor=east] {\tiny $0$};
\draw (0, 0.071428571) node[anchor=east] {\tiny $0.0$};
\draw (0, 0.142857143) node[anchor=east] {\tiny $0.5$};
\draw (0, 0.214285714) node[anchor=east] {\tiny $1.0$};
\draw (0, 0.285714286) node[anchor=east] {\tiny $1.5$};
\draw (0, 0.357142857) node[anchor=east] {\tiny $2.0$};
\draw (0, 0.428571429) node[anchor=east] {\tiny $2.5$};
\draw (0, 0.5) node[anchor=east] {\tiny $3.0$};
\foreach \y in {0, 0.083333333, 0.166666667, 0.25, 0.333333333, 0.416666667, 0.5}
  \draw (0.99,\y) -- (1.01,\y);
\draw (1, 0) node[anchor=west] {\tiny $0$};
\draw (1, 0.083333333) node[anchor=west] {\tiny $2$};
\draw (1, 0.166666667) node[anchor=west] {\tiny $4$};
\draw (1, 0.25) node[anchor=west] {\tiny $6$};
\draw (1, 0.333333333) node[anchor=west] {\tiny $8$};
\draw (1, 0.416666667) node[anchor=west] {\tiny $10$};
\draw (1, 0.5) node[anchor=west] {\tiny $12$};
\draw (0.125,0.016285714) node {\tiny $\times$};
\draw (0.1875,0.003571429) node {\tiny $\times$};
\draw (0.25,0.003857143) node {\tiny $\times$};
\draw (0.3125,0.009285714) node {\tiny $\times$};
\draw (0.375,0.010857143) node {\tiny $\times$};
\draw (0.4375,0.022142857) node {\tiny $\times$};
\draw (0.5,0.025714286) node {\tiny $\times$};
\draw (0.5625,0.044428571) node {\tiny $\times$};
\draw (0.625,0.074285714) node {\tiny $\times$};
\draw (0.6875,0.114) node {\tiny $\times$};
\draw (0.75,0.126) node {\tiny $\times$};
\draw (0.8125,0.240142857) node {\tiny $\times$};
\draw (0.875,0.321428571) node {\tiny $\times$};
\draw (0.9375,0.441) node {\tiny $\times$};
\draw (0.125,0.041666667) node {\tiny $\bullet$};
\draw (0.1875,0.083333333) node {\tiny $\bullet$};
\draw (0.25,0.083333333) node {\tiny $\bullet$};
\draw (0.3125,0.125) node {\tiny $\bullet$};
\draw (0.375,0.125) node {\tiny $\bullet$};
\draw (0.4375,0.166666667) node {\tiny $\bullet$};
\draw (0.5,0.166666667) node {\tiny $\bullet$};
\draw (0.5625,0.208333333) node {\tiny $\bullet$};
\draw (0.625,0.25) node {\tiny $\bullet$};
\draw (0.6875,0.291666667) node {\tiny $\bullet$};
\draw (0.75,0.291666667) node {\tiny $\bullet$};
\draw (0.8125,0.375) node {\tiny $\bullet$};
\draw (0.875,0.416666667) node {\tiny $\bullet$};
\draw (0.9375,0.458333333) node {\tiny $\bullet$};
\end{scope}

\begin{scope}[xshift=42,yshift=24]
    \draw (0.5,0.6) node {Surfaces in $\mathbb{P}^3$ by degree};
    \draw (0,0)--(1,0)--(1,0.5)--(0,0.5)--(0,0);
    \draw (0.5,-0.1) node {\footnotesize Degree};
    \draw (-0.14,0.25) node[rotate=90] {\footnotesize Time (seconds)};
    \draw (1.12,0.25) node[rotate=90] {\footnotesize Number of elements};
\foreach \x in {0, 0.142857143, 0.285714286, 0.428571429, 0.571428571, 0.714285714, 0.857142857, 1}
  \draw (\x,-0.01) -- (\x,0.01);
\draw (0,-0.03) node {\tiny $0$};
\draw (0.142857143, -0.03) node {\tiny $1$};
\draw (0.285714286, -0.03) node {\tiny $2$};
\draw (0.428571429, -0.03) node {\tiny $3$};
\draw (0.571428571, -0.03) node {\tiny $4$};
\draw (0.71428571, -0.03) node {\tiny $5$};
\draw (0.857142857, -0.03) node {\tiny $6$};
\draw (1,-0.03) node {\tiny $7$};
\foreach \y in {0, 0.055555556, 0.111111111, 0.166666667, 0.222222222, 0.277777778, 0.333333333, 0.388888889, 0.444444444, 0.5}
\draw (-0.01,\y) -- (0.01,\y);
\draw (0,0) node[anchor=east] {\tiny $0$};
\draw (0, 0.055555556) node[anchor=east] {\tiny $2.0$};
\draw (0, 0.111111111) node[anchor=east] {\tiny $4.0$};
\draw (0, 0.166666667) node[anchor=east] {\tiny $6.0$};
\draw (0, 0.222222222) node[anchor=east] {\tiny $8.0$};
\draw (0, 0.27777777) node[anchor=east] {\tiny $10.0$};
\draw (0, 0.333333333) node[anchor=east] {\tiny $12.0$};
\draw (0, 0.388888889) node[anchor=east] {\tiny $14.0$};
\draw (0, 0.444444444) node[anchor=east] {\tiny $16.0$};
\draw (0, 0.5) node[anchor=east] {\tiny $18.0$};
\foreach \y in {0, 0.0625, 0.125, 0.1875, 0.25, 0.3125, 0.375, 0.4375, 0.5}
  \draw (0.99,\y) -- (1.01,\y);
\draw (1, 0) node[anchor=west] {\tiny $0$};
\draw (1, 0.0625) node[anchor=west] {\tiny $2$};
\draw (1, 0.125) node[anchor=west] {\tiny $4$};
\draw (1, 0.1875) node[anchor=west] {\tiny $6$};
\draw (1, 0.25) node[anchor=west] {\tiny $8$};
\draw (1, 0.3125) node[anchor=west] {\tiny $10$};
\draw (1, 0.375) node[anchor=west] {\tiny $12$};
\draw (1, 0.4375) node[anchor=west] {\tiny $14$};
\draw (1, 0.5) node[anchor=west] {\tiny $16$};
\draw (0.285714286,0.005305556) node {\tiny $\times$};
\draw (0.428571429,0.007694444) node {\tiny $\times$};
\draw (0.571428571,0.032805556) node {\tiny $\times$};
\draw (0.714285714,0.156861111) node {\tiny $\times$};
\draw (0.857142857,0.464888889) node {\tiny $\times$};
\draw (0.285714286,0.0625) node {\tiny $\bullet$};
\draw (0.428571429,0.09375) node {\tiny $\bullet$};
\draw (0.571428571,0.15625) node {\tiny $\bullet$};
\draw (0.714285714,0.3125) node {\tiny $\bullet$};
\draw (0.85714285,0.46875) node {\tiny $\bullet$};
  \end{scope}

\begin{scope}[xshift=0,yshift=0]
    \draw (0.5,0.6) node {Threefolds in $\mathbb{P}^4$ by degree};
    \draw (0,0)--(1,0)--(1,0.5)--(0,0.5)--(0,0);
    \draw (0.5,-0.1) node {\footnotesize Degree};
    \draw (-0.14,0.25) node[rotate=90] {\footnotesize Time (seconds)};
    \draw (1.12,0.25) node[rotate=90] {\footnotesize Number of elements};
\foreach \x in {0, 0.166666667, 0.333333333, 0.5, 0.666666667, 0.833333333, 1}
  \draw (\x,-0.01) -- (\x,0.01);
\draw (0,-0.03) node {\tiny $0$};
\draw (0.166666667, -0.03) node {\tiny $1$};
\draw (0.333333333, -0.03) node {\tiny $2$};
\draw (0.5, -0.03) node {\tiny $3$};
\draw (0.666666667, -0.03) node {\tiny $4$};
\draw (0.833333333, -0.03) node {\tiny $5$};
\draw (1,-0.03) node {\tiny $6$};
\foreach \y in {0, 0.05, 0.1, 0.15, 0.2, 0.25, 0.3, 0.35, 0.4, 0.45, 0.5}
\draw (-0.01,\y) -- (0.01,\y);
\draw (0,0) node[anchor=east] {\tiny $0$};
\draw (0, 0.05) node[anchor=east] {\tiny $100$};
\draw (0, 0.1) node[anchor=east] {\tiny $200$};
\draw (0, 0.15) node[anchor=east] {\tiny $300$};
\draw (0, 0.2) node[anchor=east] {\tiny $400$};
\draw (0, 0.25) node[anchor=east] {\tiny $500$};
\draw (0, 0.3) node[anchor=east] {\tiny $600$};
\draw (0, 0.35) node[anchor=east] {\tiny $700$};
\draw (0, 0.4) node[anchor=east] {\tiny $800$};
\draw (0, 0.45) node[anchor=east] {\tiny $900$};
\draw (0, 0.5) node[anchor=east] {\tiny $1000$};
\foreach \y in {0, 0.0625, 0.125, 0.1875, 0.25, 0.3125, 0.375, 0.4375, 0.5}
  \draw (0.99,\y) -- (1.01,\y);
\draw (1, 0) node[anchor=west] {\tiny $0$};
\draw (1, 0.0625) node[anchor=west] {\tiny $5$};
\draw (1, 0.125) node[anchor=west] {\tiny $10$};
\draw (1, 0.1875) node[anchor=west] {\tiny $15$};
\draw (1, 0.25) node[anchor=west] {\tiny $20$};
\draw (1, 0.3125) node[anchor=west] {\tiny $25$};
\draw (1, 0.375) node[anchor=west] {\tiny $30$};
\draw (1, 0.4375) node[anchor=west] {\tiny $35$};
\draw (1, 0.5) node[anchor=west] {\tiny $40$};
\draw (0.333333333,0.0003975) node {\tiny $\times$};
\draw (0.5,0.0044995) node {\tiny $\times$};
\draw (0.666666667,0.0509995) node {\tiny $\times$};
\draw (0.833333333,0.453829) node {\tiny $\times$};
\draw (0.333333333,0.025) node {\tiny $\bullet$};
\draw (0.5,0.075) node {\tiny $\bullet$};
\draw (0.666666667,0.2) node {\tiny $\bullet$};
\draw (0.833333333,0.475) node {\tiny $\bullet$};
  \end{scope}

  \begin{scope}[xshift=42,yshift=0]
    \draw (0.5,0.6) node {Cubic hypersurfaces by dimension};
    \draw (0,0)--(1,0)--(1,0.5)--(0,0.5)--(0,0);
    \draw (0.5,-0.1) node {\footnotesize Dimension};
    \draw (-0.16,0.25) node[rotate=90] {\footnotesize Time (seconds)};
    \draw (1.12,0.25) node[rotate=90] {\footnotesize Number of elements};
\foreach \x in {0, 0.166666667, 0.333333333, 0.5, 0.666666667, 0.833333333, 1}
  \draw (\x,-0.01) -- (\x,0.01);
\draw (0,-0.03) node {\tiny $0$};
\draw (0.166666667, -0.03) node {\tiny $1$};
\draw (0.333333333, -0.03) node {\tiny $2$};
\draw (0.5, -0.03) node {\tiny $3$};
\draw (0.666666667, -0.03) node {\tiny $4$};
\draw (0.833333333, -0.03) node {\tiny $5$};
\draw (1,-0.03) node {\tiny $6$};
\foreach \y in {0, 0.125, 0.25, 0.375, 0.5}
\draw (-0.01,\y) -- (0.01,\y);
\draw (0,0) node[anchor=east] {\tiny $0$};
\draw (0, 0.125) node[anchor=east] {\tiny $5000$};
\draw (0, 0.25) node[anchor=east] {\tiny $10000$};
\draw (0, 0.375) node[anchor=east] {\tiny $15000$};
\draw (0, 0.5) node[anchor=east] {\tiny $20000$};
\foreach \y in {0, 0.1, 0.2, 0.3, 0.4, 0.5}
  \draw (0.99,\y) -- (1.01,\y);
\draw (1, 0) node[anchor=west] {\tiny $0$};
\draw (1, 0.1) node[anchor=west] {\tiny $5$};
\draw (1, 0.2) node[anchor=west] {\tiny $10$};
\draw (1, 0.3) node[anchor=west] {\tiny $15$};
\draw (1, 0.4) node[anchor=west] {\tiny $20$};
\draw (1, 0.5) node[anchor=west] {\tiny $25$};
\draw (0.166666667,0.000000625) node {\tiny $\times$};
\draw (0.333333333,0.000006925) node {\tiny $\times$};
\draw (0.5,0.000224975) node {\tiny $\times$};
\draw (0.666666667,0.00446425) node {\tiny $\times$};
\draw (0.833333333,0.4263077) node {\tiny $\times$};
\draw (0.166666667,0.04) node {\tiny $\bullet$};
\draw (0.333333333,0.06) node {\tiny $\bullet$};
\draw (0.5,0.12) node {\tiny $\bullet$};
\draw (0.666666667,0.16) node {\tiny $\bullet$};
\draw (0.833333333,0.46) node {\tiny $\bullet$};
\end{scope}
\end{tikzpicture}\\
\mbox{} \\
\begin{tabular}{cl}
  $\times$ & Run time for Algorithm \ref{alg:Stable} \\
  $\bullet$ & Size of the output $|P_s^F|$
\end{tabular}
\end{figure}
\end{center}

\subsection{Quintic threefolds}
\label{subsec:quintic-3folds}
Because of its significance in mathematical physics, Calabi-Yau threefolds have been intensively studied in last several decades. A smooth quintic threefold is one of the simplest kind of a Calabi-Yau threefold of Picard number one. The GIT compactification of the moduli space of smooth quintic threefolds is given by 
\[
	H_{3, 5} := \PP \rH^{0}(\PP^{4}, \cO_{\PP^{4}}(5))^* \git \SL_{5}.
\]
The stable locus is described in \cite{lakhani2010git}, but the semistable locus is not given there because of its computational complexity. By using our algorithms, we computed the list of maximal states describing the stable locus and the semistable locus. The number of maximal states for the stable locus is 38, and that for the semistable locus is 57. The running time for the stable locus is less than 15 minutes, and for the semistable locus is less than two hours. The computational output is available at \cite{sagemath-code}. 

\subsection{Cubic fivefolds}
\label{subsec:cubic-5folds}
Over an algebraically closed field $\Bbbk$ of characteristic $\ne 2$, up to isomorphism, there is only one smooth hypersurface of degree $d \le 2$. Thus, each of these moduli spaces is a point. Cubic hypersurfaces are thus the lowest degree cases that have non-trivial moduli, and they have attracted attention from many researchers. The GIT analysis of cubic threefolds is done in \cite{All03}, and for cubic fourfolds it is completed in \cite{Laz09}. The GIT stability of cubic fivefolds was investigated by Shibata in \cite{Shi14} (note that it does not appear to be peer-reviewed, and it does not give a complete geometric characterization of the non-stable locus). However, to the authors' knowledge, the semistable locus has not been published yet. 

By using the algorithms in Section \ref{sec:algorithm}, we recovered the results in \cite{All03, Laz09, Shi14}.%

\subsection{Mukai models}\label{subsec:Mukai}

In a series of papers beginning in 1992, Mukai introduced three projective GIT quotients that are birational models of the Deligne-Mumford compactification $\overline{\rM}_g$ of the moduli space of curves of genus $g$ for $7 \leq g \leq 9$. See the announcement \cite{Mukai1992survey} for an overview and \cites{Mukai1992g8,Mukai1995,Mukai2010} for details. Although nearly 30 years have passed years since these models were introduced, very little is known about their boundaries. We discuss them briefly now. 

\subsubsection{Genus 7} 

In \cite{Mukai1995} Mukai showed that the GIT quotient $\Gr(7,S^{+}) \git \Spin(10)$ is a birational model of $\overline{\rM}_7$. Here $\Spin(10)$ is a double cover of $\SO(10)$, and $S^+ \cong \Gamma_{\omega_4}$ is the 16-dimensional half-spin representation of $\Spin(10)$ with highest weight $\omega_4$. The map to $\overline{\rM}_7$ arises because the intersection of a generic 6-dimensional projective linear space with the orthogonal Grassmannian $\mathrm{OG}(5,10) \subset \PP(S^{+})^*$ is a canonically embedded genus 7 curve. The orthogonal Grassmannian is a homogeneous space for $\Spin(10)$, and moving the linear space by an element of $\Spin(10)$ does not change the isomorphism class of the curve.

This quotient corresponds to the GIT problem for the representation $\bigwedge^7 S^{+}$. We have $\binom{16}{7} = 11,440$. The group $\Spin(10)$ has rank 5. We compute $|A_3| = 852$ for Algorithm \ref{alg:Stable}. Thus $\binom{A_3}{d-1} \approx 21.8 \times 10^{9}$. We deemed this too large to run Algorithm \ref{alg:Stable} using our \texttt{SageMath} software. However, we wrote highly optimized \texttt{C++} code \cite{cpp-code} to compute a superset $\wtS_m$ of the set $\mathcal{S}_m$ using a variation of Algorithm \ref{alg:Stable} as follows.

\begin{algorithm}\label{alg:StableMukai7}
[Algorithm for the computation of $\wtS_m$]
\leavevmode\\
\textbf{Input:} The state $\Xi_{V}$.\\
\textbf{Output:} A set $\wtS_m$ containing all the maximal non-stable states.\\
\begin{enumerate}[label=\arabic*.]
\item[1--9.] Compute $A_3$ as in Algorithm \ref{alg:Stable}.
\item[10.] $\wtS_{m} := \emptyset$
\item[11.] \textbf{for all} $I \in \binom{A_{3} }{ d-1}$ \textbf{do}
\item[12.] \qquad \textbf{if} $I$ is linearly independent \textbf{then do}
\item[13.] \qquad \qquad Calculate $\lambda \ne 0$ such that $\langle \lambda, \chi\rangle = 0$ for all $\chi \in I$
\item[14.] \qquad \qquad \textbf{if} $\lambda \notin F$ \textbf{then} $\lambda := -\lambda$
\item[15.] \qquad \qquad \textbf{if} $\lambda \in F$ \textbf{then do}
\item[16.] \qquad \qquad \qquad $\wtS_{m} := \wtS_{m} \cup \{\Xi_{V, \lambda \ge 0}\}$
\item[17.] \textbf{return} $\wtS_{m}$
\end{enumerate}
\end{algorithm}

We find that $|\wtS_m| = 10,620,905$. Due to the large size of $\wtS_m$, we did not attempt to compute the maximal elements of $\wtS_m$ with respect to inclusion. 

Any attempt at geometrically analyzing the maximal non-stable states also seems doomed, due to the size of $\wtS_m$. We therefore explored other approaches to studying Mukai's model of $\overline{\rM}_7$. In the preprint \cite{Swin23}, the fourth author uses invariant theory to establish the GIT semistability of some singular curves in this GIT problem, including a 7-cuspidal curve, the genus 7 balanced ribbon, and a family of highly reducible nodal curves.

\subsubsection{Genus 8} In \cite{Mukai1992g8} Mukai showed that the GIT quotient $\Gr(8,\bigwedge^2 V) \git \SL_6$ is a birational model of $\overline{\rM}_8$. Here $V \cong \Gamma_{\omega_1}$ is the standard representation of $\SL_6$. The map to $\overline{\rM}_8$ arises because the intersection of a generic 7-dimensional projective linear space with the Grassmannian $\Gr(2,V) \subset \PP(\bigwedge^2 V)^*$ is a canonically embedded genus 8 curve. We compute $|A_3| = 739$ for Algorithm \ref{alg:Stable}. Thus $\binom{A_3}{d-1} \approx 12.3 \times 10^{9}$. Once again, we deemed this too large for our \texttt{SageMath} software. In future work we will apply our \texttt{C++} code instead.

\subsubsection{Genus 9} In \cite{Mukai2010} Mukai showed that the GIT quotient $\Gr(9,\Gamma_{\omega_3}) \git \Sp_6$ is a birational model of $\overline{\rM}_9$. Here $\Gamma_{\omega_3}$ is the irreducible representation of $\Sp_6$ with highest weight $\omega_3$. It has dimension 14. The map to $\overline{\rM}_9$ arises because the intersection of a generic 8-dimensional projective linear space with the symplectic Grassmannian $ \Sp(3,6) \subset \PP(\Gamma_{\omega_3})^*$ is a canonically embedded genus 9 curve. 

In this case, we have $|A_3| = 51$ for Algorithm \ref{alg:Stable}, and $|B_2| = 120$ for Algorithm \ref{alg:Semistable}. We used our \texttt{SageMath} software to compute $|P_{s}^F| = 142$ and $|P_{ss}^F| = 186$. The running times for these calculations were approximately 2 hours and 4 hours, respectively. We did not attempt to run Algorithm \ref{alg:Polystable} for this example.

\section{An application to K-moduli of Fano threefolds}\label{sec:Kstability}

In this section, we assume that $\Bbbk = \CC$. 

We discuss the compactification of the one-dimensional moduli space of Fano threefolds with $-K_X^3 = 32$, $h^{1,2} =1$. This is family 2.25  in the Mori-Mukai classification. The smooth member of this family of Fano varieties is obtained by blowing up $\PP^3$ along a smooth complete intersection of two quadric surfaces, i.e. a smooth elliptic quartic. By \cite[Theorem B]{GLHS18}, the main component 
$\mathrm{Hilb}^{4t}_{\tiny{main}}(\PP^3)$
of the Hilbert scheme  associated to such curves is a double  blow up of the Grassmanian
\[
\mathrm{Hilb}^{4t}_{\tiny{main}}(\PP^3)
\longrightarrow
\Gr \left(2, \rH^0(\PP^3, \cO_{\PP^{3}}(2)) \right) 
\cong 
\mathbb{G}\left(\PP^1, \PP^9\right).
\]
Two elliptic curves are isomorphic if and only if they are equivalent by a projective automorphism of $\PP^3$.
Thus, the action of the projective automorphism of $\PP^3$ lifts to the above Grassmanian and the GIT moduli space of elliptic quartics in $\PP^3$ is equal to 
\begin{equation}
\Gr \left(2, \rH^0( \PP^3, \cO_{\PP^{3}}(2)) \right)  \git  \PGL_{4} 
\cong
\Gr \left(2, \rH^0( \PP^3, \cO_{\PP^{3}}(2)) \right)  \git  \SL_{4} 
\label{eq:family2-25}    
\end{equation}
Next, we describe the GIT (semi)stability analysis in detail.  
By Theorem \ref{thm:functoriality} we can use the  Pl\"ucker embedding 
\[
\Gr \left(2, \rH^0( \PP^3, \cO_{\PP^{3}}(2)) \right) 
\longrightarrow
\PP(\bigwedge^2 \rH^0( \PP^3, \cO_{\PP^{3}}(2)))^*\cong
\PP^{\binom{10}{2}-1} \cong \PP^{44}
\]
to determine the (semi)stable locus. We denote the coordinates of the Pl\"ucker embedding as 
$X_{i} X_{j} \wedge X_{s} X_{r}$ with $i,j,s,r \in \{0,1,2,3 \}$. So, the action of a diagonal one-parameter subgroup $\lambda = (a_0, a_1, a_2, a_3)$
is
\[
\lambda(t) \cdot (X_{i} X_{j} \wedge X_{s} X_{r})
= t^{a_i+a_j+a_s+a_r} X_{i} X_{j} \wedge X_{s} X_{r}.
\]
Let 
\[
f(X_0,X_1,X_2,X_3):= 
\sum_{i \leq k}a_{i,j}X_{i}X_{k}, \qquad
g(X_0,X_1,X_2,X_3):= 
\sum_{i \leq k}b_{i, k}X_{i}X_{k}
\]
be the equations of two quadrics such that $C := \{f =g =0 \}$. The Pl\"ucker coordinates of the curve $C$ are all the $(2\times 2)$-minors of a $(2 \times 10)$ matrix $H$ given by
\[
H =
\begin{pmatrix}
a_{0,0} & a_{0,1} & a_{0,2} & a_{1,1} & a_{1,2} & a_{2,2} &
a_{0,3} & a_{1,3} &  a_{2,3} & a_{3,3} 
\\
b_{0,0} & b_{0,1} & b_{0,2} & b_{1,1} & b_{1,2} & b_{2,2} &
b_{0,3} & b_{1,3} & b_{2,3} & b_{3,3} 
\end{pmatrix}.
\]

\begin{lemma}\label{lemma:SingCurve}
The complete intersection of two quadrics $C:=Q_1 \cap Q_2 $ has a singular point if and only if up to the $\SL_4$-action, the equations of the quadrics can be written as
\begin{eqnarray*}
f_1(X_0,X_1,X_2,X_3) &=& a_0X_3 X_0 + q(X_0,X_1,X_2)
\\
f_2(X_0,X_1,X_2,X_3) &=& X_3(b_0 X_0 + b_1 X_1 +b_2 X_2) + q'(X_0,X_1,X_2), 
\end{eqnarray*}
where either $a_0 =0$ or $b_1=b_2 = 0$.
\end{lemma}

\begin{proof}
Without loss of generality, we may assume that the singular point of $C$ is $p:=[0:0:0:1]$. The condition that $p\in C$ implies that the equations of the quadrics can be written as
\begin{eqnarray*}
f_{1}(X_0,X_1,X_2,X_3) &=& X_3\ell(X_0,X_1,X_2) + q(X_0,X_1,X_2)
\\
f_2(X_0,X_1,X_2,X_3) &=& X_3 \ell'(X_0,X_1,X_2) + q'(X_0,X_1,X_2),
\end{eqnarray*}
where $\ell(X_0,X_1,X_2)$ and $\ell'(X_0,X_1,X_2)$ are linear forms while $q(X_0,X_1,X_2)$ and $q'(X_0,X_1,X_2)$ are quadratic forms. Applying a projective transformation fixing $p$, we can write 
$\ell(X_0,X_1,X_2)$ as $X_0$. Then, we can write the above equations as
\begin{eqnarray*}
f_{1}(X_0,X_1,X_2,X_3) &=& a_0 X_3 X_0 + q(X_0,X_1,X_2)
\\
f_2(X_0,X_1,X_2,X_3) &=& X_3\left( b_0 X_0 + b_1 X_1 + b_2 X_2 \right) + q'(X_0, X_1, X_2).
\end{eqnarray*}
The curve $C$ is singular at $p$ if and only if the rank of the Jacobian matrix evaluated at that point is less than two. In our particular case such matrix is given by 
\begin{eqnarray*}
\begin{pmatrix}
\nabla f_1(p)\\
\nabla f_2(p)
\end{pmatrix}
= 
\begin{pmatrix}
a_0 & 0 & 0 & 0
\\
b_0 & b_1 &  b_2 & 0
\end{pmatrix}.
\end{eqnarray*}
Its rank is less than two if and only if $a_0b_1=a_0b_2=0$.
\end{proof}

Let $V := \bigwedge^2 \rH^0( \PP^3, \cO_{\PP^{3}}(2))$.  Then $V = \Gamma_{3\omega_1+\omega_2}$ is an irreducible $\SL_4$-representation
with the highest weight $\omega = 3\omega_1 + \omega_2$. 

\subsection{Stability analysis}

Algorithm \ref{alg:Stable} gives a set of maximal non-stable sets $P^F_s$ associated to the five one parameter subgroups
\begin{align}\label{eq:ListCurveOPS}
\lambda_1 = (1, 1, 1, -3), 
&& \lambda_2 = (1,0,0,-1), 
&& \lambda_{3} = (3,1,-1,-3), 
\\
\lambda_4 =(3,-1,-1,-1),
&& \lambda_5 =(1,1,-1,-1).
& &
\notag
\end{align}

The following lemma gives a geometric characterization of the nonstable locus. We know of two different proofs of this lemma. One strategy is to do a case-by-case analysis of the outputs of our algorithms. For a different strategy, see \cite{papazachariou2022k}.

\begin{lemma}[cf.~{\cite[Theorem 4.10]{papazachariou2022k}}]
\label{lemma:Papazachariou}
The complete intersection $C$ is not stable if and only if it is singular. 
\end{lemma}
\begin{proof}
First,  suppose that $C$ is singular. By Lemma \ref{lemma:SingCurve}, up to a change of coordinates, we have 
\[
C= 
\{
a_0 X_3 X_0 + q(X_0, X_1, X_2) =
X_3(b_0 X_0 + b_1 X_1 +b_2 X_2) + q'(X_0, X_1, X_2) = 0 
\} 
\]
with either $a_0=0$ or $b_0=b_1=0$. We examine both cases and show they imply $\mu(C, \lambda_1) \geq 0$. 

Indeed, if $a_0=0$, the Pl\"ucker embedding of $C$ has nonzero coefficients only for the vectors of the form $X_3 X_i \wedge  X_s X_r $ and $X_i X_j \wedge X_s X_r$ with $i,j,s,r \in \{0,1,2\}$ and $i+s+r=3$. Then 
\[
 \lambda_1(t) \cdot X_3 X_i \wedge  X_s X_r  = t^{-3 +i+r+s}X_3 X_i \wedge  X_s X_r = X_3 X_i \wedge  X_s X_r,
 \qquad
 \lambda_1(t) \cdot X_j X_i \wedge  X_s X_r = t^4 X_j X_i \wedge  X_s X_r.
\]
So $\mu(C, \lambda_1) \ge 0$. If $b_1=b_2=0$, then a similar direct calculation shows that the Pl\"ucker coefficients are nonzero only for the same forms to the previous case. So we obtain $\mu(C, \lambda_1) \geq 0$. As a consequence, the curve is not stable by Theorem \ref{thm:HilbertMumford}.

Conversely, suppose that $C$ is not stable. The hypothesis that $C$ is not stable implies that it is projectively equivalent to a curve $C'$ whose state $\Xi_{C'}$ is contained in one of $\Xi_{V, \lambda_i \geq 0}$ with $\lambda_i$ as listed on Equation \eqref{eq:ListCurveOPS}. We present the analysis of one of the five cases,  $\Xi_{C'} \subseteq \Xi_{V, \lambda_4 \geq 0}$, below.

Algorithm \ref{alg:Stable} and its implementation gives the maximal non-stable state::
\[
\begin{split}
\Xi_{V, \lambda_4 \ge 0} =& 
\{(1, 2, 1, 0) ,(2, 0, 0, 2) ,(2, 0, 2, 0) ,(1, 0, 2, 1),(1, 1, 1, 1) ,(1, 1, 2, 0) ,(1, 0, 3, 0),\\
&(3, 0, 0, 1),(1, 3, 0, 0) ,(1, 2, 0, 1),(2, 2, 0, 0),(2, 1, 1, 0) ,(2, 1, 0, 1) ,(1, 0, 0, 3),\\
&(1, 1, 0, 2),(1, 0, 1, 2) ,(3, 1, 0, 0) ,(2, 0, 1, 1),(3, 0, 1, 0)\},
\end{split}
\]

The containment 
$\Xi_{C'} \subseteq \Xi_{V, \lambda_4 \geq 0}$ implies that the curve $C'$ can be written as 
\[
\left\{
d_0 X_3^2 + X_3\left(\sum_{i=0}^2 a_i X_i\right) + q_2(X_0,X_1,X_2) =
c_0 X_3^2 + X_3\left(\sum_{i=0}^2 b_i X_i\right) + q_2'(X_0,X_1,X_2) = 0
\right\}
\]
where $q_2(X_0,X_1,X_2)$ and $q_2'(X_0,X_1,X_2)$ are homogeneous polynomials of degree two. 

The first conclusion is that $d_0c_0=0$. Otherwise, the monomial $X_3^2$ will be present in both quadratic equations with non-zero coefficients. This will imply the existence of the character $(0,0,0,4)$ in $\Xi_{C'}$, but it does not exist in  
$\Xi_{V, \lambda_4 \geq 0}$, contradicting 
$\Xi_{C'} \subset \Xi_{V, \lambda_4 \geq 0}$.

By symmetry, we may assume that $d_0 =0$. If $c_0\ne 0$, then we have nonzero Pl\"ucker coordinates for 
$X_3^2 \wedge \prod_{i=1}^3 X_i^{m_i}$
whose associated character is $(m_0,m_1,m_2, m_3+2)$. The only such character with $m_3=1$ in $\Xi_{V, \lambda_4 \ge 0}$, is $(1, 0, 0, 3)$. Thus, the equations for $C'$ are of the form 
\[
\left\{
a_0 X_3 X_0 + q_2(X_0, X_1, X_2) =
c_0 X_3^2 + X_3\left(\sum_{i=0}^2 b_i X_i\right) + q_2'(X_0, X_1, X_2) 
=0
\right\}.
\]

Further inspection of the characters within $\Xi_{V, \lambda_4 \geq 0}$ and the last coordinate is $2$, which are $(1,1,0,2)$, $(1,0,1,2)$, and $(2,0,0,2)$, we find that the first coordinate must be nonzero. This last fact constrains the possible monomials with nonzero coefficients, and $C'$ is
\[
\left\{
X_0 f_1(X_0, X_1, X_2, X_3)
=
c_0 X_3^2 + X_3\left(\sum_{i=0}^2 b_i X_i\right) + q_2'(X_0,X_1,X_2) = 0
\right\}.
\]
Now it is straightforward to check that $C'$ is singular, as it is on the intersection of a quadric surface and a union of two planes.

The proofs for the other cases are similar.%
\end{proof}

\subsection{Polystability analysis}
Next, we discuss the polystable curve  with maximal stabilizer. In this particular example, after relabeling, we have the equation $\{X_0 X_1 = X_2 X_3 = 0\}$. Note that the associated Pl\"ucker point $X_0 X_1 \wedge X_2 X_3$ is invariant with respect to a maximal torus $T$ because the associated state is $(1, 1, 1, 1)$, which corresponds the trivial character (Recall that for the type $A_n$, $M_\RR$ can be identified with $\RR^{n+1}/(\sum \ee_i = 0)$.). The curve $C$ represents the union of four lines supported on the toric boundary of $\PP^3$. 

Thus, we conclude that a curve $C$ in this one-dimensional family $\Gr(2, \rH^0(\PP^3, \cO_{\PP^2}(2)))\git \SL_4$ is stable if and only it is smooth and it is strictly polystable if and only if it is $C_0 := \{ X_0 X_1 = X_2 X_3 = 0\}$. By blowing-up $\PP^3$ along each such curve, one can construct a one-dimensional compact family of (possibly singular) Fano threefolds over the GIT quotient \eqref{eq:family2-25}, where all smooth elements are K-stable (see \cite{CalabiFanoProject}). The singular curve $C_0$ is toric, so the blow-up $Y_0$ of $\PP^3$ along $C_0$ is a toric variety. One can check that the barycentre of its toric polytope is the origin (e.g. by running a script on Magma), which means that $Y_0$ is K-polystable. Thus, one has that \eqref{eq:family2-25} parametrizes compact family of K-polystable Fano threefolds. Now, using the inverse moduli continuity method in \cite{papazachariou2022k}, it follows that \eqref{eq:family2-25} is isomorphic to the K-moduli component of this family.

\section{Potential improvements and new problems}\label{sec:improvement}

In this section, we mention three known improvements and one conjectural improvement to the algorithms in Section \ref{sec:algorithm}, which may be worth considering for large problems. %
  Finally, we discuss three open problems for future work.

\subsection{A sufficient condition for maximality for nonstable states}

For each state $\Xi_{V, \lambda \ge 0}$ that we computed in Line 16 of Algorithm \ref{alg:Stable}, there is a sufficient condition for maximality. 

\begin{proposition}\label{prop:maximal}
If $\chi_0 \in \intr\Conv(\Xi_{V, \lambda = 0}) \subset \lambda^\perp$, then $\Xi_{V, \lambda \ge 0}$ 
is maximal in $\{\Xi_{V, \mu \ge 0}\}$. 
\end{proposition}

\begin{proof}
Suppose not. Then there is $\mu \in N_\RR$ such that $\Xi_{V, \lambda \ge 0} \subsetneq \Xi_{V, \mu \ge 0}$. In particular, $\lambda$ and $\mu$ are not proportional. Then $\Xi_{V, \lambda = 0} \cap \Xi_{V, \mu \ge 0}$ is a half-space in $\Xi_{V, \lambda = 0}$, so by the assumption, it cannot include all characters in $\Xi_{V, \lambda = 0}$. Therefore, there is a $\chi \in \Xi_{V, \lambda \ge 0} \setminus \Xi_{V, \mu \ge 0}$.  
\end{proof}

\subsection{Essential pairs and triples}\label{ssec:essentialpair}

In Line 11 of Algorithm \ref{alg:Stable}, we consider the set of all $(d-1)$-subsets of essential characters. When $d$ is large, this is expensive. One possibility to reduce the size of the set is to extend the notion of essential characters to essential subsets. 

\begin{definition}\label{def:essential}
A finite set of nontrivial characters $S := \{\chi_1, \chi_2, \ldots, \chi_k\}$ is \emph{essential} if $S$ is linearly independent and $\mathrm{Span}(S)^\perp \cap F \ne \{\textbf{0}\}$.
\end{definition}

When $S = \{\chi\}$ is a singleton set, $S$ is essential if and only if $\chi \in \bigcup_{i=1}^d \Xi_{V, \gamma_i \ge 0} \setminus \bigcap_{i=1}
^d \Xi_{V, \gamma_i > 0}$ (Proposition \ref{prop:differenceforstability}). If $T \subset S$ and $S$ is essential, then $T$ is essential. 

A computation of essential pairs is relatively easy. By definition, a pair $\{\chi_1, \chi_2\}$ is essential if and only if $\mathrm{Span}(\chi_1, \chi_2)^\perp \cap F \ne \{\chi_0\}$ and $\chi_1$ and $\chi_2$ are not proportional. Because $F$ is a full-dimensional strongly convex cone, this is equivalent to the condition that for the projection map 
\[
    \phi : N_\RR \stackrel{\left(\begin{array}{c}\chi_1 \\ \chi_2 \end{array}\right)}{\longrightarrow} \RR^2,
\]
$\phi(F) = \RR^2$. It occurs if and only if 
\[
    \textbf{0} \in \intr \Conv(\phi(\gamma_1), \phi(\gamma_2), \ldots, \phi(\gamma_d)).
\]
Since this is a convex geometry computation in two dimensional space, the verification is quick. And we expect that the set of essential pairs is very small compare to $\binom{A_3}{2}$.

Note that in Algorithm \ref{alg:Stable}, to make $\cS_m$, instead of using $\binom{A_3}{d-1}$, it suffices to use the proper subset of essential $(d-1)$-sets. Any $(d-1)$ essential set can be obtained by taking a union of $\lceil \frac{d-1}{2}\rceil$ essential pairs. 

A similar approach is possible for Algorithm \ref{alg:Semistable}. For the semistability, we need to use the following definition. 

\begin{definition}\label{def:essentialforss}
A finite set of nonzero characters $S = \{\chi_1, \chi_2, \ldots, \chi_k\}$ is \emph{essential}
\begin{enumerate}
\item if $S = \{\chi\}$, then $\chi \in \bigcup_{1=1}^d \Xi_{\gamma_i > 0} \setminus K_{nm}$ (Lemma \ref{lem:optimization2});
\item if $|S| > 1$, then the set of vectors in $N_\RR$ that is perpendicular to the affine space generated by $S$, which is a sub vector space of $N_\RR$, intersects $F$ nontrivially. 
\end{enumerate}
\end{definition}

Note that in Line 16 of Algorithm \ref{alg:Semistable}, instead of $\binom{B_2 }{ d}$, we only need to take the set of all essential $d$-subsets. Note also that if $S$ is essential and $T \subset S$, then $T$ is also essential. Furthermore, if $|S| = 3$, then we can obtain the following criterion for essentiality --- $\{\chi_1, \chi_2, \chi_3\}$ is an essential triple if and only if for the projection 
\[
    \phi : N_\RR \stackrel{\left(\begin{array}{c}\chi_1 - \chi_2\\ \chi_1 - \chi_3\end{array}\right)}{\longrightarrow} \RR^2, 
\]
$\phi(F) = \RR^2$, or equivalently, $\textbf{0} \in \intr \Conv(\phi(\gamma_1), \phi(\gamma_2), \ldots, \phi(\gamma_d))$. Now every essential $d$-sets can be obtained by taking a union of $\lceil \frac{d}{3}\rceil$ of essential triples.

\subsection{Parallel computation and existing algorithms to find maximal sets}

Several steps in algorithms \ref{alg:Stable}, \ref{alg:Semistable}, and \ref{alg:Polystable} can be parallelized, allowing the answers to be computed more quickly.  For example, lines 12--16 in Algorithm \ref{alg:Stable} can be executed for each set $I \in \binom{A_3}{d-1}$ in parallel, and  lines 17--20 in Algorithm \ref{alg:Semistable} can be executed for each set $I \in \binom{B_2}{d}$ in parallel.

Finding maximal elements of a given set of states can also be performed in parallel (cf. lines \texttt{18-21,23-28} in Algorithm \ref{alg:Stable} and \texttt{22-25,29-32} in Algorithm \ref{alg:Semistable}); indeed, this is a well-researched problem. Given a collection $\mathcal F$ of subsets $S_1, \ldots, S_m$ over some  common domain (which in our case is almost always $\Xi_V$ or a subset of it), one chooses $N=\sum |S_i|$ to be the problem size and considers finding the maximal elements in $\mathcal F$.  In \cite{Yellin-Jutla}, the authors provided an algorithm requiring $O(N^2/\log N)$ dictionary operations with worst-case running time of $O(N^2/\sqrt{\log N})$. %

Nonetheless, the real bottleneck for mathematicians applying these algorithms will not be in the computation of states or the finding of maximal states, but in the interpretation of the outputs ($P^F_s, P^F_{ss}, P^F_{ps}$) in geometric terms. Indeed, in our experience, such interpretation for \emph{one} state in any of these sets takes significantly longer than executing several times the algorithms that produced them. Thus, until significant improvement takes place in the automatic recognition of singularities and invariants of families of varieties produced by our algorithms, the potential optimizations at implementation level discussed above will mean very little.

\subsection{A conjecture about the Weyl group action}
Algorithms \ref{alg:Stable} and \ref{alg:Semistable} each use the Weyl group symmetry in two different ways. First, the Weyl group symmetry is used to significantly reduce the set of one-parameter subgroups we need to consider. Then in the last stage of each of these algorithms (lines 23-28 of Algorithm \ref{alg:Stable}, or lines 27-32 of Algorithm \ref{alg:Semistable}) we apply the action of $W$ (or a subset $W'$ of $W$) to make the output non-redundant.  These lines are intended to remove nonstable (respectively unstable) states that are not maximal because they are subsets of some nonstable (resp. unstable) maximal state in another Weyl chamber.  

\textit{A priori}, such a containment seems possible to us. However, in all the examples that we have run so far \cite{sagemath-code}, we have not seen such a containment occur; that is, the last optimization step does not make any difference to the output. This leads to the following conjecture.

\begin{conjecture} \label{conj: Wprime optimization}
    At the last step of the algorithms, the optimization routine using $W' \subset W$ does not reduce the output of algorithms \ref{alg:Semistable} and \ref{alg:Stable}.
\end{conjecture}

Proving Conjecture \ref{conj: Wprime optimization} would allow us to remove these steps from Algorithms \ref{alg:Stable} and \ref{alg:Semistable}, improving their speed.

\subsection{On a question from a 2004 workshop at AIM}
\label{ssec:1psApproach}

The following question was posed at the workshop `Compact moduli spaces and birational geometry' at the American Institute of Mathematics in 2004.  
\begin{question}[{\cite[Problem 3.2]{Van04}}]
   \emph{``For hypersurfaces of a given dimension [$n$] and degree [$d$], is there a bound on the exponents appearing in the diagonal 1-PS that need to be checked?''}.
\end{question}

The existence of such a bound follows immediately from the finiteness of the sets $\Lambda_{ss}$ and $\Lambda_{s}$ of Corollary \ref{cor:finite}, but the true intention of \cite[Problem 3.2]{Van04} is to give a explicit estimate in terms of $n$ and $d$, preferably one that is sharp or nearly sharp. We will not give a thorough solution to this problem here, but we want to point out that the ideas used to develop Algorithms \ref{alg:Stable} and \ref{alg:Semistable} can be used to give a coarse upper bound.

Consider the stable locus. (A similar discussion applied to the semistable locus.) By Algorithm \ref{alg:Stable}, it is enough to consider 1-parameter subgroups $\lambda$ that are orthogonal to each element in a subset $I$ of $A_3$ having size $(d-1)$. Such a $\lambda$ can be expressed using the cofactors of the $(d-1)\times d$ matrix whose rows consist of the characters $\chi \in I$. Then any bound on these cofactors (for instance, Hadamard's Inequality, $|\det A | \leq \prod_{j=1}^{n} \|A_j\| $, where $A$ is $n\times n$ and $A_j$ is the $j^{th}$ column) leads to a bound on the coefficients of $\lambda$. However, this can be far from sharp. For example, for cubic surfaces, Hadamard's Inequality gives 9 as the bound on each cofactor. But the output from Algorithm \ref{alg:Stable} shows that it suffices to work with one-parameter subgroups with coefficients in $\{0,\pm 1,\pm 2\}$.

\subsection{Variation of GIT quotients}

Recall the original setting of a polarized pair $(X, L)$ with a reductive linearized $G$-action on it. If $\mathrm{rank} \mathrm{Pic}(X) \ge 2$ or $G$ has a torus factor, there are many possible linearizations, and different linearizations can give rise to non-isomorphic GIT quotients $X\git_L G$ \cite{Tha96, DH98}.%

Most of the work on computational VGIT has focused on the case of an affine variety modulo a torus (\cite{Kei12, BKR20}), rather than a projective variety modulo a noncommutative group. %
  In \cite{GMG18} the first two authors introduced the notion of compactification of the moduli space of log pairs formed by a Fano or Calabi-Yau hypersurface $X_d\subset \PP^n$ of degree $d$ and a hyperplane section using VGIT quotients by the group $\SL_{n+1}$. They also provided algorithms in the spirit of Algorithms \ref{alg:Stable}, \ref{alg:Semistable}, \ref{alg:Polystable} (although less complete, efficient and only for certain choice of group) and demonstrated their use to describe VGIT compactifications of the moduli space of log pairs formed by a cubic surface and an anti-canonical divisor in \cite{GMGS21}. It would be interesting  to extend the algorithms studied here to study VGIT of semisimple/reductive groups acting on $X=\PP^{n_1}\times \cdots \times \PP^{n_k}$, with the goal of describing both the VGIT wall-and-chamber decomposition as well as the (semi/poly)stable points within each chamber. Then, using Theorem \ref{thm:functoriality}, one may be able to extend this description to a more general $X$ (e.g a Mori dream space, where the space of stability conditions is polyhedral).

\subsection{On the limits of functoriality}

In theory, Theorem \ref{thm:functoriality} is sufficient to determine the stability of any $X$ with a linearized reductive group action $G$. But the suggested algorithms in this paper are not efficient enough to deduce $X^{ss}(L) = X \cap \PP V^{* ss}(\cO(1))$ and $X^{s}(L) = X \cap \PP V^{* s}(\cO(1))$, as many states (as a subset of $\Xi_V$) are not realized as $\Xi_x$ for some point $x \in X$. 

Many natural explicit parameter spaces are given by the Grassmanianns $\Gr(k, V)$. Instead of using its Pl\"ucker embedding $\Gr(k, V) \subset \PP (\wedge^k V)^*$ and applying Algorithms \ref{alg:Stable}, \ref{alg:Semistable}, and \ref{alg:Polystable}, it is desirable to find an algorithm that directly calculates $P_s^F, P_{ss}^F$, and $P_{ps}^F$ from $\Gr(k, V)$. Combining our ideas and \cite{papazachariou2022k} to describe GIT quotients of Grassmannians by simple groups algorithmically may be possible and it may have applications to moduli theory.

\printbibliography

\end{document}